\documentclass[11pt,reqno]{amsart}
\usepackage{amssymb}
\usepackage{amsmath}
\usepackage{amsthm}
\usepackage{graphicx}
\usepackage{caption}
\usepackage{bm}
\usepackage[colorlinks=true,urlcolor=blue,
citecolor=red,linkcolor=blue,linktocpage,pdfpagelabels,
bookmarksnumbered,bookmarksopen]{hyperref}%per colore-link
\usepackage[titletoc,title]{appendix}
\numberwithin{equation}{section} %riparte da zero ogni sezione
\newcounter{cont}[section] 

\newtheorem{thm}[cont]{Theorem}
\newtheorem{prop}[cont]{Proposition}
\newtheorem{lem}[cont]{Lemma}

\theoremstyle{definition}

 \theoremstyle{remark}
 \newtheorem{rem}[cont]{Remark}

\newcommand{\R}{\mathbb{R}}
\newcommand{\e}{\varepsilon}
\newcommand{\ut}{\tilde{u}}
\newcommand{\wt}{\tilde{w}}
\newcommand{\vt}{\tilde{v}}
\begin{document}
\baselineskip=16pt

\title[Metastability for the hyperbolic Cahn--Hilliard equation]{Metastability and layer dynamics for the hyperbolic relaxation of the Cahn--Hilliard equation}

\author[R. Folino]{Raffaele Folino}
\address[Raffaele Folino]{Dipartimento di Ingegneria e Scienze dell'Informazione e Matematica, Universit\`a degli Studi dell'Aquila (Italy)}
\email{raffaele.folino@univaq.it}
	
\author[C. Lattanzio]{Corrado Lattanzio}
\address[Corrado Lattanzio]{Dipartimento di Ingegneria e Scienze dell'Informazione e Matematica, Universit\`a degli Studi dell'Aquila (Italy)}
\email{corrado@univaq.it} 
 
\author[C. Mascia]{Corrado Mascia}
\address[Corrado Mascia]{Dipartimento di Matematica,
Sapienza Universit\`a di Roma (Italy)}
\email{mascia@mat.uniroma1.it}


\maketitle

%\tableofcontents

\begin{abstract} 
The goal of this paper is to accurately describe the metastable dynamics of the solutions to the hyperbolic relaxation of the Cahn--Hilliard equation in a bounded interval 
of the real line, subject to homogeneous Neumann boundary conditions.
We prove the existence of an \emph{approximately invariant manifold} $\mathcal{M}_0$ for such boundary value problem, that is
we construct a narrow channel containing $\mathcal{M}_0$ and satisfying the following property:
a solution starting from the channel evolves very slowly and leaves the channel only after an exponentially long time.
Moreover, in the channel the solution has a \emph{transition layer structure} and 
we derive a system of ODEs, which accurately describes the slow dynamics of the layers.
A comparison with the layer dynamics of the classic Cahn--Hilliard equation is also performed.
\end{abstract}

\section{Introduction}
The celebrated \emph{Cahn--Hilliard equation},
\begin{equation*}
	u_t=\Delta\left(-\e^2\Delta u+F'(u)\right), 
\end{equation*} 
where $\e$ is a positive constant and $F:\R\rightarrow\R$ is a double well potential with wells of equal depth,
was originally proposed in \cite{Cahn-Hilliard} to model phase separation in a binary system at a fixed temperature,
with constant total density and where $u$ stands for the concentration of one of the two components.
Among the phase transformations involved in phase separation, a peculiar one is named ``spinodal decomposition", 
which indicates the stage during which the mixture quickly becomes inhomogeneous, forming a fine-grained structure (cfr. \cite{Cahn,Grant2,Mai-Wan}).
In order to model the early stages of spinodal decomposition in certain glasses,
some physicists \cite{Galenko,GalenkoJou,LeZaGa} proposed the following \emph{hyperbolic relaxation of the Cahn--Hilliard equation} 
\begin{equation}\label{eq:hyp-CH-multiD}
	\tau u_{tt}+u_t=\Delta\left(-\e^2\Delta u+F'(u)\right),
\end{equation}
where $\tau$ is a positive constant. 
In particular, the \emph{hyperbolic version} \eqref{eq:hyp-CH-multiD} has been firstly proposed by Galenko in \cite{Galenko}, 
following the classical Maxwell--Cattaneo modification of the Fick's diffusion law \cite{Cat}.
Many papers have been devoted to the study of the dynamics of the solutions to \eqref{eq:hyp-CH-multiD}.
Without claiming to be complete, we list the following papers:
for the long-time behavior of the solutions and the limiting behavior as $\tau\to0$ in the one-dimensional case 
see \cite{Debussche,ZhengMilani2004,ZhengMilani2005,BonGraMir} and references therein;
for the multidimensional-case among others, we mention \cite{GraSchZel,GraSchZel2}.

In this paper, we are interested in studying the metastable dynamics of the solutions to the one-dimensional version of \eqref{eq:hyp-CH-multiD} 
\begin{equation}\label{eq:hyp-CH}
	\tau u_{tt}+u_t=\left(-\e^2u_{xx}+F'(u)\right)_{xx}, \qquad \qquad x\in(0,1),\, t>0,
\end{equation}
subject to homogeneous Neumann boundary conditions
\begin{equation}\label{eq:Neumann}
	u_x(0,t)=u_x(1,t)=u_{xxx}(0,t)=u_{xxx}(1,t)=0, \qquad \forall\,t\geq0,
\end{equation}
and initial data
\begin{equation}\label{eq:initial}
	u(x,0)=u_0(x), \quad u_t(x,0)=u_1(x), \qquad x\in[0,1].
\end{equation}
Precisely, we are interested in describing the behavior of the solutions to the initial boundary value problem \eqref{eq:hyp-CH}-\eqref{eq:Neumann}-\eqref{eq:initial} 
when the parameter $\e$ is very small and the function $F\in C^4(\R)$ satisfies 
\begin{equation}\label{eq:ass-F}
	F(\pm1)=F'(\pm1)=0, \qquad F''(\pm1)>0 \qquad \mbox{ and } \qquad F(u)>0, \quad \mbox{ for } \, u\neq\pm1.
\end{equation}
The simplest example of function satisfying \eqref{eq:ass-F} is $F(u)=\frac14(u^2-1)^2$.

The existence and persistence for an exponentially large time of metastable states with $N$ transitions between $-1$ and $+1$ 
for the IBVP \eqref{eq:hyp-CH}-\eqref{eq:initial} has been proved in \cite{FLM17pre}
by using an energy approach firstly introduced in \cite{Bron-Kohn} to study the \emph{Allen--Cahn equation}
\begin{equation}\label{eq:AC}
	u_t=\e^2u_{xx}-F'(u),
\end{equation}
and subsequently used in \cite{Bron-Hilh} to prove existence of metastable states for the classic Cahn--Hilliard equation
\begin{equation}\label{eq:CH}
	u_t=\left(-\e^2u_{xx}+F'(u)\right)_{xx}, \qquad \qquad x\in(0,1),\, t>0.
\end{equation}
Here, we investigate the metastable properties of the solutions to \eqref{eq:hyp-CH} by using a different approach,
the dynamical approach proposed by Carr and Pego in \cite{Carr-Pego} and Fusco and Hale \cite{Fusco-Hale} to study the Allen--Cahn equation \eqref{eq:AC} 
and used for the Cahn--Hilliard equation \eqref{eq:CH} in \cite{AlikBateFusc91} and \cite{Bates-Xun1,Bates-Xun2}.
The dynamical approach gives a more precise description of the dynamics of the solution to the IBVP \eqref{eq:hyp-CH}-\eqref{eq:initial}
and allows us to derive a system of ODEs which describes the evolution of such solution.

Before presenting our results, let us briefly describe the dynamics of the solutions to the classic Cahn--Hilliard equation \eqref{eq:CH}
with homogeneous Neumann boundary conditions \eqref{eq:Neumann} and recall some previous results on the metastable behavior of the solutions.
First of all, notice that any constant function is a equilibrium solution to \eqref{eq:CH}-\eqref{eq:Neumann}
and that, by integrating the equation \eqref{eq:CH} and using the boundary conditions \eqref{eq:Neumann} one finds out that the total mass $\displaystyle\int_a^b u(x,t)\,dx$ is conserved.
A linear analysis of the equation \eqref{eq:CH} about a constant solution
shows that spatially homogeneous equilibria in the \emph{spinodal region} (where $F''<0)$ are unstable \cite{Grant2}.
Moreover, it is sufficient to take an initial datum which is a small perturbation of a fixed constant in the spinodal region and the corresponding solution exhibits the phenomenon of \emph{spinodal decomposition}:
after a relatively short time, the solution to \eqref{eq:CH}-\eqref{eq:Neumann} is approximately close to $+1$ or $-1$ (the positions of the global minimum of $F$) except near a finite number $N$ of transition layers.
The first mathematical treatment and rigorous verification of such phenomenon is performed in \cite{Grant2}.
After the spinodal decomposition, the solution, which has a $N$-\emph{transition layers structure}, evolves so slow that the profile \emph{appears} to be stable.
On the other hand, it is well-known that the Cahn--Hilliard equation \eqref{eq:CH} possesses the Lyapunov functional
\begin{equation}\label{eq:energy}
	E_\e[u]=\int_a^b\left[\frac{\e^2}2u_x^2+F(u)\right]\,dx,
\end{equation}
and the solutions converge as $t\to+\infty$ to a stationary solution \cite{Zheng1986}.
The problem to minimize the energy functional \eqref{eq:energy} among all the functions satisfying $\displaystyle\int_a^b u\,dx=M$ (the total mass being conserved),
has been investigated in \cite{CarrGurtSlem} for the one-dimensional case and in \cite{Modica} for the multi-dimensional case.
In particular, in \cite{CarrGurtSlem} it has been proved that if $\e$ is small enough and $M\in(-1,1)$, then
all the minimizers are strictly monotone functions.
Therefore, the solution to \eqref{eq:CH}-\eqref{eq:Neumann} converges to a limit with a single transition and, as a consequence, we have an example of \emph{metastable dynamics}:
the solution maintains the (unstable) $N$-transitions layer structure for a very long time $T_\e$ and then converges to the asymptotic limit with a single transition.
Precisely, the evolution of the solutions depends only from the interactions between the layers, which move with an exponentially small velocity as $\e\to0$;
it follows that the lifetime $T_\e$ of a metastable state with $N$ transitions is exponentially large as $\e\to0$,
namely $T_\e=\mathcal O\left(e^{C/\e}\right)$ where $C>0$ depends only on $F$ and on the distance between the layers.

As it was previously mentioned, there are at least two different approaches to study the metastable dynamics of the solutions, 
which have been proposed in the study of the Allen--Cahn equation \eqref{eq:AC}.
The energy approach of \cite{Bron-Kohn} is based on $\Gamma$-convergence properties of the functional \eqref{eq:energy} and 
it has been applied to the Cahn--Hilliard equation \eqref{eq:CH} in \cite{Bron-Hilh};
it permits to handle both Neumann \eqref{eq:Neumann} 
and Dirichlet boundary conditions of the type
\begin{equation}\label{eq:Dirichlet}
	u(0,t)=\pm1, \quad u(1,t)=\pm1 \qquad \mbox{ and } \qquad u_{xx}(0,t)=u_{xx}(1,t)=0, \qquad \forall\,t\geq0.
\end{equation}
On the other hand, the dynamical approach of \cite{Carr-Pego,Fusco-Hale} is performed in \cite{AlikBateFusc91}, 
where the authors consider the case of an initial datum with a $2$-transition layer structure and in \cite{Bates-Xun1,Bates-Xun2},
where the general case of $N+1$ layers ($N\geq1$) is considered.
This approach permits to describe in details the movement of the layers.
In the two layer case, for the conservation of the mass, we have that the layers move in an almost rigid way (they move in the same direction at approximately the same \emph{exponentially small} velocity);
when the layers are more than $2$ the situation is more complicated and we will describe their dynamics in Section \ref{sec:layers}.

Each of the previous approaches has its advantages and drawbacks. 
The dynamical approach gives more precise results: it gives the exact order of the speed of the slow motion, and allows us to accurately describe the movement of the layers,
but permits to study only the case of homogeneous boundary conditions and the proofs are complicated and lengthy. 
The energy approach is fairly simple, it provides a rather clear and intuitive explanation for the slow motion and permits to handle both Neumann \eqref{eq:Neumann}
and Dirichlet \eqref{eq:Dirichlet} boundary conditions, but it gives only an upper bound for the velocity of the layers.
We also recall that the energy approach permits to study the vector-valued version of \eqref{eq:CH}, that is
when $u$ takes value in $\mathbb{R}^m$ and the potential $F$ vanishes only in a finite number of points (for details, see \cite{Grant}).
Finally, we mention that both the dynamical and the energy approach can be applied to study the metastability for the following hyperbolic variations of the Allen--Cahn equation
\begin{equation}\label{eq:hypAC}
	\tau u_{tt}+ g(u)u_t=\e^2u_{xx}-F'(u),
\end{equation}
for any positive function $g\in C^1(\mathbb{R})$ (cfr. \cite{Folino,Folinoproc,Folino2,FLM17}).

In this paper, we apply the dynamical approach to the IBVP \eqref{eq:hyp-CH}-\eqref{eq:Neumann}-\eqref{eq:initial}.
The well-posedness and the asymptotic behavior as $t\to+\infty$ of the solutions to such IBVP are investigated in \cite{Debussche}. 
A fundamental difference with respect to the classic Cahn--Hilliard equation \eqref{eq:CH} is that the homogeneous Neumann boundary conditions \eqref{eq:Neumann}
do not imply conservation of the mass;
as we will see in Section \ref{sec:prelimin} the solution to the IBVP \eqref{eq:hyp-CH}-\eqref{eq:Neumann}-\eqref{eq:initial}
conserves the mass if and only if the initial velocity $u_1$ is of zero mean.
Therefore, to apply the dynamical approach we need a further assumption on $u_1$;
however, by using the energy approach, it is possible to prove the metastable dynamics of the solutions without the assumption of zero-mean for $u_1$
(for details see \cite[Remark 2.7]{FLM17pre}).

The main idea of the dynamical approach introduced by Carr and Pego in \cite{Carr-Pego} is to construct a family of functions $u^{\bm h}$, which approximates a metastable states with $N+1$ transitions 
located at $\bm h=(h_1,h_2,\dots,h_{N+1})$, consider the decomposition 
\begin{equation}\label{eq:decom}
	u(x,t)=u^{\bm h(t)}(x)+w(x,t),
\end{equation}
for the solution $u$ and study the evolution of the remainder function $w$ and of the transition points $h_1,h_2,\dots,h_{N+1}$.
By inserting the decomposition \eqref{eq:decom} in the equation \eqref{eq:AC} and imposing an orthogonality condition on $w$,
it is possible to derive an ODE-PDE coupled system for $(\bm h,w)$ and prove that the solution $u$ is well-approximated by $u^{\bm h}$ as $\e\to0$ and evolves very slowly
until either two transition points are close enough or a transition point is close enough to the boundary points of the interval $(0,1)$.
In other words, with the dynamical approach it is possible to prove the existence of an \emph{approximately invariant manifold} $\mathcal{M}$, 
consisting of functions with $N+1$ transitions between $-1$ and $+1$:
if the initial datum $u_0$ is in a particular tubular neighborhood of $\mathcal{M}$, then the transition points move with an exponentially small velocity
and the solution remains in such neighborhood for an exponentially long time.
Then, since the remainder $w$ is very small as $\e\to0$, by using the approximation $w\approx0$ in \eqref{eq:decom}
one can derive a system of ODEs for $\bm h$, which accurately describes the movement of the layers, and so the evolution of the solution $u$,
until the transition points are well-separated and far away from $0$ and $1$.

This strategy has been applied to the integrated version of \eqref{eq:hyp-CH}-\eqref{eq:Neumann} in \cite{Bates-Xun1,Bates-Xun2}
and gives a precise description of the metastable dynamics of the solutions.
In the following, we will show how to adapt this strategy to the hyperbolic version \eqref{eq:hyp-CH} and we will analyze the differences with respect to \eqref{eq:CH}.
In particular, we will prove the existence of an approximately invariant manifold $\mathcal{M}_0$ contained in a narrow channel for the initial boundary problem \eqref{eq:hyp-CH}-\eqref{eq:Neumann}-\eqref{eq:initial}:
if the initial datum \eqref{eq:initial} is in the channel, then the solution $u$ remains in the channel for an exponentially long time.
Moreover, in the channel the following estimates hold:
\begin{equation}\label{eq:estimates-intro}
	\|u-u^{\bm h}\|_{{}_{L^\infty}}\leq C\e^{-5/2}\exp\left(-\frac{Al^{\bm h}}{\e}\right), \qquad \qquad |\bm h'|_{{}_\infty}\leq C\e^{-2}\tau^{-1/2}\exp\left(-\frac{Al^{\bm h}}{\e}\right),
\end{equation}
where $A:=\sqrt{\min\{F''(-1),F''(+1)\}}$, $\ell^{\bm h}:=\min\{h_j-h_{j-1}\}$ and $|\cdot|_{{}_{\infty}}$ denotes the maximum norm in $\mathbb{R}^N$.
Furthermore, we will derive the following system of ODEs
\begin{equation}\label{eq:ODE-intro}
	\tau \bm h''+\bm h'+\tau\bm{\mathcal{Q}}(\bm h,\bm h')=\bm{\mathcal P}(\bm h),
\end{equation}
which describes the movement of the transition layers and has to be compared with the system $\bm h'=\bm{\mathcal P}(\bm h)$, 
which describes the dynamics of the solutions to the classic Cahn-Hilliard equation;
for the formulas of $\bm{\mathcal P}$ and $\bm{\mathcal{Q}}$ see Section \ref{sec:layers}.

The rest of the paper is organized as follows.
In Section \ref{sec:prelimin} we give all the definitions, the preliminaries and we construct the approximate invariant manifold $\mathcal{M}_0$
following the ideas of \cite{Carr-Pego}, \cite{Bates-Xun1}.
Section \ref{sec:slow} contains the main result of the paper, Theorem \ref{thm:main}, where we prove that the manifold $\mathcal{M}_0$
is approximately invariant for \eqref{eq:hyp-CH}-\eqref{eq:Neumann}, by constructing a slow channel which contains $\mathcal{M}_0$ and where
the solutions stay for an exponentially long time and satisfy \eqref{eq:estimates-intro}.
Finally, Section \ref{sec:layers} is devoted to the description of the movement of the layers. 
We will derive the system of ODEs \eqref{eq:ODE-intro} and we will analyze the differences between the classic Cahn--Hilliard equation \eqref{eq:CH} and its hyperbolic relaxation \eqref{eq:hyp-CH}.

\section{Preliminaries}\label{sec:prelimin}
In this section we collect some results of \cite{Carr-Pego}, \cite{Bates-Xun1,Bates-Xun2} we will use later
and we introduce the \emph{extended base manifold} $\mathcal{M}_0$, which is, as we shall prove in Section \ref{sec:slow}, 
approximately invariant for the boundary value problem \eqref{eq:hyp-CH}-\eqref{eq:Neumann}.
\subsection{Approximate metastable states}
The aim of this subsection is to construct a family of functions with $N+1$ transitions between $-1$ and $+1$, approximating metastable states for \eqref{eq:hyp-CH}.
Such construction was firstly introduced by Carr and Pego \cite{Carr-Pego} to describe the metastable dynamics of the solutions to the Allen--Cahn equation \eqref{eq:AC},
and then it has also been used to study the metastability for the Cahn--Hilliard equation \cite{Bates-Xun1,Bates-Xun2} 
and for hyperbolic variants of the Allen--Cahn equation \cite{FLM17}.
Here, we briefly recall the construction of the family and some useful properties we will use later, for details see \cite{Carr-Pego}.

For fixed $\rho>0$, we introduce the set
\begin{equation*}
	\Omega_\rho:=\bigl\{{\bm h}\in\mathbb{R}^{N+1}\, :\,0<h_1<\cdots<h_{N+1}<1,\quad
		 h_j-h_{j-1}>\varepsilon/\rho\mbox{ for } j=1,\dots,N+2\bigr\},
\end{equation*}
where we define $h_0:=-h_1$ and $h_{N+2}:=2-h_{N+1}$, because of the homogeneous Neumann boundary conditions \eqref{eq:Neumann}.
In what follows, we fix a minimal distance $\delta\in(0,1/(N+1))$ and we consider the parameters $\varepsilon$ and $\rho$ such that
\begin{equation}\label{eq:triangle}
	0<\varepsilon<\varepsilon_0\qquad\textrm{and}\qquad \delta<\frac{\varepsilon}{\rho}<\frac{1}{N+1},
\end{equation}
for some $\varepsilon_0>0$ to be chosen appropriately small.

We associate to any $\bm h\in\Omega_\rho$ a function $u^{\bm h}=u^{\bm h}(x)$ which approximates a metastable
state with $N+1$ transition points located at $h_1,\dots,h_{N+1}$. 
To do this, we make use of the solutions to the following boundary value problem:
given $\ell>0$, let $\phi(\cdot,\ell,+1)$ be the solution to
\begin{equation}\label{eq:fi}
	\mathcal{L}^{AC}(\phi):=-\varepsilon^2\phi_{xx}+F'(\phi)=0, \qquad \quad
	\phi\bigl(-\tfrac12\ell\bigr)=\phi\bigl(\tfrac12\ell\bigr)=0,
\end{equation}
with $\phi>0$ in $(-\tfrac12\ell,\tfrac12\ell)$,
and $\phi(\cdot,\ell,-1)$ the solution to \eqref{eq:fi} with $\phi<0$ in $(-\tfrac12\ell,\tfrac12\ell)$.
The functions $\phi(\cdot,\ell,\pm1)$ are well-defined if $\ell/\varepsilon$ is sufficiently large, and they depend on $\varepsilon$ 
and $\ell$ only through the ratio $\varepsilon/\ell$. 
Moreover, we have  
\begin{equation*}
	\max_x|\phi(\cdot,\ell,\pm1)|=|\phi(0,\ell,\pm1)|=M_\pm(\ell/\varepsilon)
	\qquad \quad \textrm{and} \qquad \quad
	\max_x|\phi_x(\cdot,\ell,\pm1)|\leq C\varepsilon^{-1},
\end{equation*}
where $C>0$ is a constant depending only on the function $F$. 
In particular, $M_\pm$ tends to $+1$ as $\varepsilon/\ell\to 0$ (more details in Proposition \ref{prop:alfa,beta}).

The family of the approximate metastable states is constructed by matching together 
the functions $\phi(\cdot,\ell,\pm1)$, using smooth cut-off functions:
given $\chi:\mathbb{R}\rightarrow[0,1]$ a $C^\infty$-function with $\chi(x)=0$ for $x\leq-1$ and $\chi(x)=1$ for $x\geq1$, set 
\begin{equation*}
	\chi^j(x):=\chi\left(\frac{x-h_j}\varepsilon\right) \qquad\textrm{and}\qquad
	\phi^j(x):=\phi\left(x-h_{j-1/2},h_j-h_{j-1},(-1)^j\right),
\end{equation*}
where 
\begin{equation*}
	h_{j+1/2}:=\tfrac12(h_j+h_{j+1})\qquad j=0,\dots,N+1,
\end{equation*}
(note that $h_{1/2}=0$, $h_{N+3/2}=1$).
Then, we define the function $u^{\bm h}$ as
\begin{equation}\label{eq:uh}
	u^{\bm h}:=\left(1-\chi^j\right)\phi^j+\chi^j\phi^{j+1} \qquad \textrm{in}\quad I_j:=[h_{j-1/2},h_{j+1/2}],
\end{equation}
for $j=1,\dots,N+1$, and the manifold
\begin{equation*}
	\mathcal{M}^{AC}:=\{u^{\bm h} :\bm h\in\Omega_\rho\}.
\end{equation*} 
In \cite{Carr-Pego}, the authors show that the manifold $\mathcal{M}^{AC}$ is approximately invariant for the Allen--Cahn equation \eqref{eq:AC}.
On the other hand, the \emph{extended} manifold
\begin{equation*}
	\mathcal{M}^{AC}_{{}_0}:=\mathcal{M}^{AC}\times\{0\}=\{(u^{\bm h},0) :u^{\bm h}\in\mathcal{M}^{AC}\}
\end{equation*}
is approximately invariant for the hyperbolic variant \eqref{eq:hypAC}, see \cite{FLM17}.

To get an idea of the structure of the function $u^{\bm h}$ defined in \eqref{eq:uh}, 
we recall that, if $\rho>0$ is sufficiently small and $\bm h\in\Omega_\rho$, then $u^{\bm h}\approx\pm1$ away from $h_j$ for $j=1,\dots,N+1$,
and $u^{\bm h}(x)\approx\Phi\left((x-h_j)(-1)^{j-1}\right)$ for $x$ near $h_j$, where $\Phi$ is the unique solution to the problem
\begin{equation*}
	\mathcal{L}^{AC}(\Phi):=-\varepsilon^2\Phi_{xx}+F'(\Phi)=0, \qquad \qquad \lim_{x\rightarrow\pm\infty} \Phi(x)=\pm\infty, \qquad\qquad \Phi(0)=0.
\end{equation*}
For instance, in the case $F(u)=\frac14(u^2-1)^2$, the unique solution is $\Phi(x)=\tanh(x/\sqrt2\e)$.
In conclusion, we say that $u^{\bm h}$ is a smooth function of $\bm h$ and $x$, 
which is approximately $\pm1$ except near $N+1$ transition points located at $h_1,\cdots,h_{N+1}$;
moreover, $\mathcal{L}^{AC}(u^{\bm h})=0$ except in an $\e$--neighborhood of the transition points $h_j$.
Precisely, we have
\begin{equation}\label{eq:prop-uh}
	\begin{aligned}
	u^{\bm h}(0)&=\phi(0,2h_1,-1)<0,
			&\qquad 	u^{\bm h}(h_{j+1/2})&=\phi\left(0,h_{j+1}-h_j,(-1)^{j+1}\right),\\
	u^{\bm h}(h_j)&=0,
			&\qquad \mathcal{L}^{AC}(u^{\bm h}(x))&=0\quad \textrm{for }|x-h_j|\geq\varepsilon,
	\end{aligned}
\end{equation}
for any $j=1,\dots,N+1$.

Central to the study of the metastable dynamics of the solutions to both the Allen--Cahn and the Cahn--Hilliard equation is an accurate characterization of the 
quantities $u^{\bm h}(h_{j+1/2})=\phi\left(0,h_{j+1}-h_j,(-1)^{j+1}\right)$ and $F\left(u^{\bm h}(h_{j+1/2})\right)$, 
because the motion of the transition points $h_1,\dots,h_{N+1}$ depend essentially on these quantities.
Since $\phi(0,\ell,\pm1)$ depends only on the ratio $r=\varepsilon/\ell$, we can define
\begin{equation*}
	\alpha_\pm(r):=F(\phi(0,\ell,\pm1)), \qquad \quad \beta_\pm(r):=1\mp\phi(0,\ell,\pm1).
\end{equation*}
By definition, $\phi(0,\ell,\pm1)$ is close to $+1$ or $-1$ and so, $\alpha_\pm(r), \beta_\pm(r)$ are close to $0$. 
The next result characterizes the leading terms in $\alpha_\pm$ and $\beta_\pm$ as $r\to 0$.
\begin{prop} [Carr--Pego \cite{Carr-Pego}] \label{prop:alfa,beta}
Let $F$ be such that \eqref{eq:ass-F} holds and set 
\begin{equation*}
	A_\pm^2:=F''(\pm1), \qquad K_{\pm}=2\exp\left\{\int_0^1\left(\frac{A_\pm}{(2F(\pm t))^{1/2}}-\frac{1}{1-t}\right)\,dt\right\}.
\end{equation*}
There exists $r_0>0$ such that if $0<r<r_0$, then
\begin{equation*}
	\begin{aligned}
	\alpha_\pm(r)&=\tfrac12K^2_\pm A^2_\pm\,\exp(-{A_\pm}/r\bigr)\bigl\{1+O\left(r^{-1} \exp(-{A_\pm}/2r)\right)\bigr\},\\
	\beta_\pm(r)&=K_\pm\,\exp\bigl(-{A_\pm}/2r\bigr)\bigl\{1+O\left(r^{-1} \exp(-{A_\pm}/2r)\right)\bigr\},
	\end{aligned}
\end{equation*}
with corresponding asymptotic formulae for the derivatives of $\alpha_\pm$ and $\beta_\pm$.
\end{prop}
For $j=1,\dots,N+1$, we set
\begin{equation*}
	l_j:=h_{j+1}-h_{j}, \qquad \qquad r_{j}:=\frac{\varepsilon}{l_j},
\end{equation*}
and
\begin{equation*}
	\alpha^{j}:=\left\{\begin{aligned}
		&\alpha_+(r_{j}) 	&j \textrm{ odd},\\
		&\alpha_-(r_{j})  	&j \textrm{ even},\\
		\end{aligned}\right.
	\qquad
	\beta^{j}:=\left\{\begin{aligned}
		&\beta_+(r_{j})	&j \textrm{ odd},\\
		&\beta_-(r_{j})	&j \textrm{ even}.\\
		\end{aligned}\right.
\end{equation*}
\begin{rem}\label{rem:alfa}
Let $\bm h\in\Omega_\rho$ with $\e,\rho$ satisfying \eqref{eq:triangle} and let $l^{\bm h}:=\min\{h_j-h_{j-1}\}$.
Then, the quantities $\alpha^j$ and $\beta^j$ are exponentially small in $\e$, namely there exists $C>0$ (independent of $\e$) such that
\begin{align}
	0<\alpha^j\leq C\exp\left(-\frac{Al_j}{\e}\right)\leq C\exp\left(-\frac{Al^{\bm h}}{\e}\right), \label{eq:alfaj}\\
	0<\beta^j\leq C\exp\left(-\frac{Al_j}{2\e}\right)\leq C\exp\left(-\frac{Al^{\bm h}}{2\e}\right), \label{eq:betaj}
\end{align}
where $A:=\sqrt{\min\{F''(-1),F''(+1)\}}$.
Moreover, assuming that $F$ is an even function and so that $\alpha_+\equiv\alpha_-$,
from Proposition \ref{prop:alfa,beta} we get
\begin{equation*}
	\frac{\alpha^j}{\alpha^i}\leq C\exp\left(-\frac{A}{\e}(l_j-l_i)\right),  
\end{equation*}
for some $C>0$.
Hence, if $l_j-l_i\geq \kappa$ for some $\kappa>0$, we deduce
\begin{equation}\label{eq:alfa^j-alfa^i}
	\alpha^j\leq C\exp\left(-\frac{A\kappa}{\e}\right)\alpha^i.
\end{equation}
Therefore, if $l_j>l_i$ then $\alpha^j<\alpha^i$, and for $\e/\kappa\ll1$, $\alpha^j$ is \emph{exponentially small} with respect to $\alpha^i$.
\end{rem}
Now, let us introduce the \emph{barrier function}
\begin{equation}\label{eq:barrier}
	\Psi(\bm h):=\sum_{j=1}^{N+1}{\langle\mathcal{L}^{AC}\bigl(u^{\bm h}\bigr),k^{\bm h}_j\rangle}^2=\sum_{j=1}^{N+1}\bigl(\alpha^{j+1}-\alpha^{j}\bigr)^2,
\end{equation}
where $\mathcal{L}^{AC}$ is the Allen--Cahn differential operator introduced above and the functions $k^{\bm h}_j$ are defined by
\begin{equation*}
	k^{\bm h}_j(x):=-\gamma^j(x)u^{\bm h}_x(x), \qquad  \mbox{ with } \; 
	\gamma^j(x):=\chi\left(\frac{x-h_{j}-\varepsilon}\varepsilon\right)\left[1-\chi\left(\frac{x-h_{j+1}+\varepsilon}\varepsilon\right)\right].
\end{equation*}
By construction, $k^{\bm h}_j$ are smooth functions of $x$ and $\bm h$ and are such that
\begin{equation*}
	\begin{aligned}
	k^{\bm h}_j(x)&=0				&\quad \textrm{for}\quad &x\notin[h_{j-1/2},h_{j+1/2}],\\
	k^{\bm h}_j(x)&=-u^{\bm h}_x(x)	&\quad \textrm{for}\quad &x\in[h_{j-1/2}+2\varepsilon,h_{j+1/2}-2\varepsilon]. 
	\end{aligned}
\end{equation*}
Such functions are fundamental in the study of the metastability for the Allen--Cahn equation \eqref{eq:AC} and for the hyperbolic Allen--Cahn equation \eqref{eq:hypAC} (see \cite{Carr-Pego} and \cite{FLM17}, respectively),
and play a crucial role in the study of the metastability for the hyperbolic Cahn--Hilliard equation \eqref{eq:hyp-CH}.
We recall that there exists $C>0$ independent of $\e$ such that
\begin{equation}\label{eq:ineq-k}
	\|k_j^{\bm h}\|+\e\|k^{\bm h}_{ij}\|\leq C\e^{-1/2}, \qquad \qquad \mbox{where } \quad
	k^{\bm h}_{ji}:=\partial_{h_i} k^{\bm h}_j.
\end{equation}
For the proof of \eqref{eq:ineq-k} see \cite[Proposition 2.3]{Carr-Pego}.

In conclusion, we collect some useful properties of the derivative of $u^{\bm h}$ with respect to $h_j$ we will use later; 
we will use the notation
\begin{equation*}
	u^{\bm h}_j:=\frac{\partial u^{\bm h}}{\partial h_j}, 
	%\qquad \quad \nabla_{\bm h} u^{\bm h}:=\bigl(u^{\bm h}_1,\dots,u^{\bm h}_{N+1}\bigr),
	\qquad \quad u^{\bm h}_{ji}:=\frac{\partial^2 u^{\bm h}}{\partial h_j\partial h_i}.
\end{equation*}
\begin{lem}\label{lem:u^h_j}
The interval $[h_{j-1}-\varepsilon,h_{j+1}+\varepsilon]$ contains the support of $u^{\bm h}_j$ and
\begin{equation*}
	u^{\bm h}_j(x)=\begin{cases}
		\mathcal{O}\left(\e^{-1}\beta^{j-1}\right), & x\in[h_{j-1}+\e, h_{j-1/2}], \\
		-u^{\bm h}_x(x)+\mathcal{O}\left(\e^{-1}\max(\beta^{j-1},\beta^{j})\right), \qquad & x\in I_j,\\
		\mathcal{O}\left(\e^{-1}\beta^{j}\right), & x\in [h_{j+1/2},h_{j+1}+\e], \\
		0, & \mbox{otherwise},
		\end{cases}
\end{equation*}
for $j=1,\dots,N+1$.
Moreover, there exists $C>0$ such that
\begin{equation}\label{eq:uh_j}
	\e\|u^{\bm h}_j\|_{{}_{L^\infty}}+\e^{1/2} \|u_j^{\bm h}\|\leq C, \qquad \qquad j=1,\dots,N+1.
\end{equation}	
\end{lem}
For the precise formula for $u^{\bm h}_j$ and the proof of Lemma \ref{lem:u^h_j} see \cite[Sections 7-8]{Carr-Pego}.

Thanks to Lemma \ref{lem:u^h_j}, we can state that if we neglect the exponentially small terms, then $u^{\bm h}_j$ is equal to $-u^{\bm h}_x$ in $I_j$ and it is zero for $x\notin I_j$.
We will use such approximation in Section \ref{sec:layers} to derive the ODE describing the motion of the transition layers $h_1,\dots,h_{N+1}$.

\subsection{Base Manifold}
In this subsection we define the base manifold for the hyperbolic Cahn--Hilliard equation \eqref{eq:hyp-CH}.
Integrating the equation \eqref{eq:hyp-CH} in $[0,1]$ and using the homogeneous Neumann boundary conditions \eqref{eq:Neumann}, 
we obtain that the total mass $m(t):=\displaystyle\int_0^1u(y,t)\,dy$ satisfies the ODE
\begin{equation*}
	\tau m''(t)+m'(t)=0, \qquad m(0)=\int_0^1 u_0(y)\,dy, \quad m'(0)=\int_0^1u_1(y)\,dy.
\end{equation*}
Then, as a trivial consequence, $m(t)=m(0)+\tau m'(0)(1-e^{-t/\tau})$ and the total mass $m$ is conserved if and only if 
\begin{equation}\label{eq:ass-u1}
	\int_0^1 u_1(y)\,dy=0.
\end{equation}
From now on, we will assume that the initial velocity satisfies \eqref{eq:ass-u1} in order to have conservation of the mass,
and we also assume that the initial profile $u_0$ has mass equal to $M$, for some $M\in(-1,1)$.
It follows that 
\begin{equation}\label{eq:conservation}
	m(t)=\int_0^1 u_0(y)\,dy=M\in(-1,1), \qquad \qquad \mbox{ for any } t\geq0.
\end{equation}
Since the total mass is conserved, we introduce the manifold
\begin{equation*}
	\mathcal{M}^{CH}:=\left\{u^{\bm h}\in \mathcal{M}^{AC}: \displaystyle\int_0^1u^{\bm h}(x)\,dx=M\right\}.
\end{equation*}
In \cite{Bates-Xun1,Bates-Xun2}, the authors study the dynamics of the solutions to the Cahn--Hilliard equation \eqref{eq:CH} in a neighborhood of $\mathcal{M}^{CH}$
and show that such manifold is approximately invariant for \eqref{eq:CH}.
In this paper, we will show that the \emph{extended base manifold}
\begin{equation*}
	\mathcal{M}^{CH}_{{}_{0}}:=\mathcal{M}^{CH}\times\{0\}=\left\{(u^{\bm h},0)\,: \, u^{\bm h}\in\mathcal{M}^{CH} \right\}
\end{equation*}
is approximately invariant for the hyperbolic Cahn--Hilliard equation \eqref{eq:hyp-CH}.
From now on, we drop the superscript $CH$ and we use the notation $\mathcal{M}^{CH}=\mathcal{M}$ and $\mathcal{M}^{CH}_{{}_{0}}=\mathcal{M}_{{}_{0}}$.

The following lemma of \cite{Bates-Xun1} is crucial in the study of the metastable dynamics of \eqref{eq:hyp-CH}
in a neighborhood of $\mathcal{M}_{{}_{0}}$. For reader's convenience, we report here below its proof.
\begin{lem}\label{lem:M(h)}
Let $M(\bm h):=\displaystyle\int_0^1u^{\bm h}(x)\,dx$, for $\bm h\in\Omega_\rho$.
Then, $M(\bm h)$ is a smooth function of $\bm h$ and 
\begin{equation*}
	\frac{\partial M}{\partial h_j}=2(-1)^{j}+ \mathcal{O}\left(\e^{-1}\max(\beta^{j-1},\beta^j)\right).
\end{equation*}
\end{lem}
\begin{proof}
By differentiating the function $M(\bm h)$ with respect to $h_j$, and by using Lemma \ref{lem:u^h_j}, we infer
\begin{align*}
	\frac{\partial M}{\partial h_j}&=\int_0^1u^{\bm h}_j(x)\,dx=-\int_{I_j}u^{\bm h}_x(x)\,dx+\mathcal{O}\left(\e^{-1}\max(\beta^{j-1},\beta^{j})\right)\\
	&=u^{\bm h}(h_{j-1/2})-u^{\bm h}(h_{j+1/2})+\mathcal{O}\left(\e^{-1}\max(\beta^{j-1},\beta^{j})\right).
\end{align*}
for $j=1,\dots,N+1$. 
From \eqref{eq:prop-uh} and the definition of $\beta^j$, it follows that
\begin{equation}\label{eq:uh(h_1/2)}
	u^{\bm h}(h_{j+1/2})=(-1)^{j+1}+(-1)^{j}\beta^j, \qquad \qquad j=0,\dots,N+1.
\end{equation}
Therefore, we can conclude that
\begin{equation*}
	\frac{\partial M}{\partial h_j}=2(-1)^j+\mathcal{O}\left(\e^{-1}\max(\beta^{j-1},\beta^{j})\right),
\end{equation*}
and the proof is complete.
\end{proof}
The previous lemma shows that the manifold $\mathcal{M}$ can be parameterized by the first $N$ components $(h_1,\dots, h_N)$ of $\bm h$.
Indeed, if $u^{\bm h}\in\mathcal{M}$, applying Lemma \ref{lem:M(h)} and the implicit function theorem, 
we can think $h_{N+1}$ as a function of $(h_1,\dots, h_N)$, namely there exists $g:\mathbb{R}^N\rightarrow\mathbb R$ such that
\begin{equation*}
	h_{N+1}=g(h_1,\dots,h_N),
\end{equation*}
and we have
\begin{equation}\label{eq:N+1-der}
	\begin{aligned}
			\frac{\partial h_{N+1}}{\partial h_j}&=-\frac{\partial M/\partial h_j}{\partial M/\partial h_{N+1}}=
		-\frac{2(-1)^{j}+\mathcal{O}\left(\e^{-1}\max(\beta^{j-1},\beta^j)\right)}{2(-1)^{N+1}+\mathcal{O}\left(\e^{-1}\max(\beta^{N},\beta^{N+1})\right)}\\
		&=(-1)^{N-j}+\mathcal{O}\left(\e^{-1}\exp\left(-\frac{Al^{\bm h}}\e\right)\right),
	\end{aligned}
\end{equation}
where we used \eqref{eq:betaj}.
Hence, we introduce the new variable $\bm\xi$, consisting of the first $N$ components of $\bm h$, and we will denote $u^{\bm h}\in\mathcal{M}$ by $u^{\bm\xi}$.
Moreover, we denote by $\mathbf{G}:\mathbb{R}^N\rightarrow\mathbb{R}^{N+1}$ the function $\mathbf{G}(\bm\xi)=(\xi_1,\dots,\xi_N,g(\xi_1,\dots,\xi_N))$,
and in the following  we will interchangeably use $\bm\xi$ and $\bm h$, meaning $\bm h=\mathbf{G}(\bm\xi)$.
Finally, we have that
\begin{equation*}
	u^{\bm\xi}_j:=\frac{\partial u^{\bm\xi}}{\partial\xi_j}=\frac{\partial u^{\bm h}}{\partial h_j}+\frac{\partial u^{\bm h}}{\partial h_{N+1}}\frac{\partial h_{N+1}}{\partial h_j},
\end{equation*}
for $j=1,\dots,N$, and using \eqref{eq:uh_j} we get
\begin{equation}\label{eq:uxi_j}
	\e\|u^{\bm\xi}_j\|_{{}_{L^\infty}}+\e^{1/2} \|u_j^{\bm\xi}\|\leq C, \qquad \qquad j=1,\dots,N.
\end{equation}

Following the previous works  \cite{AlikBateFusc91, Bates-Xun1,Bates-Xun2} on the metastability for the classic Cahn--Hilliard equation \eqref{eq:CH},
we will consider an integrated version of \eqref{eq:hyp-CH}.
If $u$ is a solution to \eqref{eq:hyp-CH} with homogeneous Neumann boundary conditions \eqref{eq:Neumann} and initial data \eqref{eq:initial}, with $u_1$ satisfying \eqref{eq:ass-u1}, 
then $\ut(x,t):=\displaystyle\int_0^x u(y,t)\,dy$ satisfies the integrated hyperbolic Cahn--Hilliard equation
\begin{equation}\label{eq:integrated-CH}
	\tau \ut_{tt}+\ut_t=-\e^2\ut_{xxxx}+F'(\ut_x)_x, \qquad \qquad x\in(0,1), \; t>0,
\end{equation}
with initial data
\begin{equation*}
	\ut(x,0)=\ut_0(x), \qquad \ut_t(x,0)=\ut_1(x), \qquad \qquad x\in[0,1],
\end{equation*}
and Dirichlet boundary conditions
\begin{equation}\label{eq:bound-ut}
	\ut(0,t)=0, \quad \ut(1,t)=M, \qquad \ut_{xx}(0,t)=\ut_{xx}(1,t)=0, \qquad \forall\,t\geq0,
\end{equation}
where $M\in(-1,1)$ is the total mass of the solution.
Here and in all the paper we use the following notation: 
given a function $u:[0,1]\rightarrow\R$, we denote by $\ut:[0,1]\rightarrow\R$ the function
\begin{equation*}
	\ut(x):=\int_0^x u(y)\,dy.
\end{equation*}
Rewrite \eqref{eq:integrated-CH} as the system
\begin{equation}\label{eq:integrated-CH-sys}
	\begin{cases}
	\ut_t=\vt,\\
	\tau \vt_t=\mathcal{L}(\ut)-\vt,
	\end{cases}
\end{equation}
where we introduced the integrated Cahn--Hilliard differential operator 
$$
	\mathcal{L}(\ut):=-\e^2 \ut_{xxxx}+\left(F'(\ut_x)\right)_x.
$$
Observe that 
\begin{equation}\label{eq:AC-CH}
	\mathcal{L}(\ut)=-\frac{d}{dx}\mathcal{L}^{AC}(u),  
\end{equation}
where $\mathcal{L}^{AC}(u):=\e^2 u_{xx}-F'(u)$ is the Allen--Cahn differential operator introduced in the previous subsection.
Since $\mathcal{L}^{AC}(u^{\bm h})=0$ except in an $\e$--neighborhood of the transition points $h_j$ (see \eqref{eq:prop-uh}),
we have that the same property holds for $\mathcal{L}(\ut^{\bm\xi})$, namely
\begin{equation*}
	\mathcal{L}(\ut^{\bm\xi}(x))=0, \qquad \qquad \mbox{ for } |x-h_j|>\e, \qquad j=1,\dots,N+1.
\end{equation*}
More precisely, one can prove that (see \cite[Lemma 5.1]{Bates-Xun1})
\begin{equation}\label{eq:Lut^xi}
	\|\mathcal{L}(\ut^{\bm\xi})\|\leq C\e^{-1}\sum_{j=1}^{N+1}|\alpha^{j+1}-\alpha^j|\leq C\e^{-1}\exp\left(-\frac{Al^{\bm h}}{\e}\right),
\end{equation}
for some positive constant $C$. 
Hence, the $L^2$--norm of $\mathcal{L}(\ut^{\bm\xi})$ is exponentially small in $\e$.

We will study the dynamics of the solutions to \eqref{eq:hyp-CH}-\eqref{eq:Neumann} in a neighborhood of $\mathcal{M}_{{}_{0}}$ by considering the integrated version \eqref{eq:integrated-CH-sys} 
and using the decomposition $\ut=\ut^{\bm\xi}+\wt$, where $u^{\bm\xi}\in\mathcal{M}^{CH}$ is defined in \eqref{eq:uh} and $\wt\in H$ for
\begin{equation}\label{eq:H}
	H:=\left\{\wt\in H^4(0,1)\, : \, \wt=\wt_{xx}=0 \, \mbox{ at } x=0,1 \mbox{ and } \; \langle\wt,E_j^{\bm\xi}\rangle=0, \; \mbox{ for } j=1,\dots,N \right\},
\end{equation}
where $E_j^{\bm\xi}$ are linear functions of $\ut^{\bm h}_j$ and $\ut^{\bm h}_{j+1}$ to be determined later.
By using the formula of Lemma \ref{lem:u^h_j}, Proposition \ref{prop:alfa,beta} and Remark \ref{rem:alfa}, we obtain
\begin{equation*}
	 \ut_j^{\bm h}(x):=\int_0^xu^{\bm h}_j(y)\,dy=\begin{cases}
	 0, & x\leq h_{j-3/2},\\
	 \mathcal{O}(e^{-c/\e}), & x\in I_{j-1},\\
	 -u^{\bm h}(x)+u^{\bm h}(h_{j-1/2})+\mathcal{O}(e^{-c/\e}), \qquad \qquad &x\in I_j,\\
	 -u^{\bm h}(h_{j+1/2})+u^{\bm h}(h_{j-1/2})+\mathcal{O}(e^{-c/\e}), & x\geq h_{j+1/2},
	 \end{cases}
\end{equation*}
for $j=1,\dots, N$, and 
\begin{equation*}
	 \ut_{N+1}^{\bm h}(x):=\int_0^xu^{\bm h}_{N+1}(y)\,dy=\begin{cases}
	 0, & x\leq h_{N-1/2},\\
	  \mathcal{O}(e^{-c/\e}), & x\in I_{N},\\
	 -u^{\bm h}(x)+u^{\bm h}(h_{N+1/2})+\mathcal{O}(e^{-c/\e}), \qquad \qquad &x\in I_{N+1}.
	 \end{cases}
\end{equation*}
Here and in what follows $c$ is a generic positive constant independent on $\e$.
Using \eqref{eq:uh(h_1/2)}, we deduce
\begin{equation}\label{eq:ut-j-h}
	 \ut_j^{\bm h}(x)=\begin{cases}
	 0, & x\leq h_{j-3/2},\\
	 \mathcal{O}(e^{-c/\e}), & x\in I_{j-1},\\
	 -u^{\bm h}(x)+(-1)^j+\mathcal{O}(e^{-c/\e}), \qquad \qquad &x\in I_j,\\
	 2(-1)^j+\mathcal{O}(e^{-c/\e}), & x\geq h_{j+1/2},
	 \end{cases}
\end{equation}
for $j=1,\dots, N$, and 
\begin{equation}\label{eq:ut-N+1-h}
	 \ut_{N+1}^{\bm h}(x)=\begin{cases}
	 0, & x\leq h_{N-1/2},\\
	  \mathcal{O}(e^{-c/\e}), & x\in I_{N},\\
	 -u^{\bm h}(x)+(-1)^{N+1}+\mathcal{O}(e^{-c/\e}), \qquad \qquad &x\in I_{N+1}.
	 \end{cases}
\end{equation}
Since $\ut^{\bm\xi}_j=\ut_j^{\bm h}+\ut_{N+1}^{\bm h}\frac{\partial h_{N+1}}{\partial h_j}$, for \eqref{eq:N+1-der}, \eqref{eq:ut-j-h} and \eqref{eq:ut-N+1-h}, we get
\begin{equation}\label{eq:ut-j-xi}
	 \ut_j^{\bm \xi}(x)=\begin{cases}
	 0, & x\leq h_{j-3/2},\\
	  \mathcal{O}(e^{-c/\e}),  & x\in I_{j-1},\\
	 -u^{\bm\xi}(x)+(-1)^j+\mathcal{O}(e^{-c/\e}),  & x\in I_j,\\
	 2(-1)^j+\mathcal{O}(e^{-c/\e}),  & x\in[h_{j+1/2},h_{N+1/2}],\\
	 -u^{\bm\xi}(x)(-1)^{N-j}+(-1)^j+\mathcal{O}(e^{-c/\e}), \qquad & x\in I_{N+1},
	 \end{cases}
\end{equation}
for $j=1,\dots,N$.
Let $\omega_j:=\ut_j^{\bm h}+\ut_{j+1}^{\bm h}$, $j=1,\dots,N$; one has  
\begin{equation}\label{eq:omega_j}
	 \omega_j(x)=\begin{cases}
	 0, & x\leq h_{j-3/2},\\
	 \mathcal{O}(e^{-c/\e}), & x\in I_{j-1},\\
	 -u^{\bm\xi}(x)+(-1)^j+\mathcal{O}(e^{-c/\e}), \qquad \qquad &x\in I_j\cup I_{j+1},\\
	  \mathcal{O}(e^{-c/\e}), & x\geq h_{j+3/2}.
	 \end{cases}
\end{equation}
Then, the functions $\omega_j$ are either zero or exponentially small outside of $I_j\cup I_{j+1}$.
Now, we can define the functions $E^{\bm\xi}_j$ introduced above:
\begin{equation}\label{eq:Ej-Qj}
	E^{\bm\xi}_j(x):=\omega_j(x)-Q_j(x),
\end{equation}
where
\begin{equation*}
	Q_j(x):=\left(-\tfrac{1}{6}x^3+\tfrac12x^2-\tfrac13x\right)\omega''_j(0)+\tfrac16(x^3-x)\omega''_j(1)+x\omega_j(1).
\end{equation*}
As it was shown in \cite[Section 3, formula (54)]{Bates-Xun1}, the terms $\omega''_j(0)$, $\omega''_j(1)$ and $\omega_j(1)$ are exponentially small as $\e\to0^+$.
Hence, $Q_j$ are exponentially small functions introduced so that $E_j^{\bm\xi}$ satisfies
\begin{equation*}
	E^{\bm\xi}_j(x)=E^{\bm\xi}_{jxx}(x)=0, \qquad \mbox{ for } x=0,1, \qquad j=1,\dots,N.
\end{equation*}
The functions $E_j^{\bm\xi}$ are good approximations of the first $N$ eigenfunctions of the eigenvalue problem
\begin{align*}
	L^{\bm\xi} u:=&-\e^2u_{xxxx}+\left(F''(u^{\bm\xi})u_x\right)_x=\lambda u, \qquad & \mbox{ in } (0,1),\\
	&u(x)=u''(x)=0, \qquad & \mbox{ for } x=0,1,
\end{align*}
where $L^{\bm\xi}$ is the linearized operator of $\mathcal{L}$ about $\ut^{\bm\xi}$.
Indeed, in \cite{Bates-Xun1} it is proved that $L^{\bm\xi}$ has $N$ exponentially small eigenvalues and that all the others are bounded away from zero uniformly in $\e$ (see \cite[Theorem A]{Bates-Xun1}).
From \eqref{eq:omega_j} and the fact that $Q_j$ are exponentially small functions, we obtain 
\begin{equation}\label{eq:Ej}
	E^{\bm\xi}_j(x)=\begin{cases}
	-u^{\bm\xi}(x)+(-1)^j+\mathcal{O}(e^{-c/\e}), \qquad \qquad & x\in I_j\cup I_{j+1},\\
	 \mathcal{O}(e^{-c/\e}), & \mbox{otherwise},
	\end{cases}
\end{equation}
for $i=j,\dots,N$.

We conclude this section recalling that the existence of the coordinate system $\ut=\ut^{\bm\xi}+\wt$ with $\wt\in H$ in a neighborhood of $\mathcal{M}$ 
has been proved in \cite[Theorem A.7]{Bates-Xun2}.
For Lemma \ref{lem:M(h)} and the subsequent comments, in the following result we can use $\bm\xi$ and $\bm h$ interchangeably.
\begin{thm}\label{thm:existence-coord}
There exists $\rho_0>0$ such that if $\rho\in(0,\rho_0)$ and $\ut$ satisfies 
\begin{equation*}
	\ut(0)=\ut_{xx}(0)=\ut_{xx}(1)=0, \quad \ut(1)=M, \qquad  \mbox{and } \qquad \|\ut-\ut^{\bm k}\|_{{}_{L^\infty}}\leq \e^2
\end{equation*}
for some $\bm k\in\Omega_\rho$, then there is a unique $\bar{\bm{h}}\in\Omega_\rho$ such that
\begin{equation*}
	\ut=\ut^{\bar{\bm{h}}}+\wt, \qquad \mbox{ with } \qquad \langle\wt,E_j^{\bm\xi}\rangle=0, \;\; j=1,\dots,N.
\end{equation*}
Moreover, if $\|\ut-\ut^{\bm{h}^*}\|_{{}_{L^\infty}}=\inf\{ \|\ut-\ut^{\bm k}\|_{{}_{L^\infty}}\, :\, \bm k\in\Omega_\rho\}$ for some $\bm{h}^*\in\Omega_\rho$,
then there exists a positive constant $C$ such that
\begin{equation*}
	|\bar{\bm{h}}-\bm{h}^*|\leq C\|\ut-\ut^{\bm{h}^*}\|_{{}_{L^\infty}}, \qquad \mbox{ and } \qquad \|\ut-\ut^{\bar{\bm{h}}}\|_{{}_{L^\infty}}\leq C\|\ut-\ut^{\bm{h}^*}\|_{{}_{L^\infty}}.
\end{equation*}
\end{thm}

\section{Slow dynamics in a neighborhood of the base manifold}\label{sec:slow}
The aim of this section is to study the dynamics of the solutions to \eqref{eq:hyp-CH}-\eqref{eq:Neumann} in a neighborhood of the manifold $\mathcal{M}_{{}_0}$
and to prove that $\mathcal{M}_{{}_0}$ is approximately invariant for \eqref{eq:hyp-CH}-\eqref{eq:Neumann}.
To do this, we will consider the integrated version \eqref{eq:integrated-CH}-\eqref{eq:bound-ut}. 
Since
\begin{equation}\label{eq:u-ut}
	\|\ut\|_{{}_{L^\infty}}\leq\|u\|_{{}_{L^\infty}},
\end{equation}
if $\|u-u^{\bm h}\|_{{}_{L^\infty}}$ is sufficiently small for some $\bm h\in\Omega_\rho$, we can use Theorem \ref{thm:existence-coord} and the decomposition $\ut=\ut^{\bm h}+\wt$ introduced in Section \ref{sec:prelimin}.

\subsection{Equations of motion and slow channel}
Let $(\ut,\vt)$ be a solution to \eqref{eq:integrated-CH-sys} with $\ut=\ut^{\bm\xi}+\wt$ and $\wt\in H$, where $H$ is the space defined in \eqref{eq:H};
it follows that the variables $(\wt,\vt)$ satisfy
\begin{equation*}
	\begin{cases}
	\wt_t=\vt-\displaystyle\sum_{j=1}^N\ut_j^{\bm\xi}\xi'_j,\\
	\tau\vt_t=\mathcal{L}(\ut^{\bm\xi}+\wt)-\vt.
	\end{cases}
\end{equation*}
Expanding we get
\begin{equation*}
	\mathcal{L}(u^{\bm\xi}+\wt)=\mathcal{L}(\ut^{\bm\xi})+L^{\bm\xi}\wt+(f_2\wt_x^2)_x,
	\qquad\textrm{ where }\quad
	f_2:=\int_0^1(1-s)F'''(\ut^{\bm\xi}_x+s\wt_x)\,ds,
\end{equation*}
and $L^{\bm\xi}$ is the linearized operator of $\mathcal{L}$ about $\ut^{\bm\xi}$, that is $L^{\bm\xi}\wt:=-\e^2\wt_{xxxx}+\left(F''(u^{\bm\xi})\wt_x\right)_x$.
Hence, we obtain the following system for $(\wt,\vt)$:
\begin{equation}\label{eq:w-v}
	\begin{cases}
	\wt_t=\vt-\displaystyle\sum_{j=1}^N\ut_j^{\bm\xi}\xi'_j,\\
	\tau\vt_t=\mathcal{L}(\ut^{\bm\xi})+L^{\bm\xi}\wt+(f_2\wt_x^2)_x-\vt.
	\end{cases}
\end{equation}
In order to obtain the equation for $\bm\xi=\bm\xi(t)$, we make use of the orthogonality condition
\begin{equation}\label{eq:orthogonality}
	\langle \wt,E^{\bm\xi}_j\rangle=0, \qquad \mbox{ for } j=1,\dots,N,
\end{equation}
where the functions $E_j^{\bm\xi}$ are defined in Section \ref{sec:prelimin} and satisfy \eqref{eq:Ej}.
By differentiating with respect to $t$ the conditions \eqref{eq:orthogonality} and by using the first equation of \eqref{eq:w-v}, we infer
\begin{equation}\label{eq:xi}
	\langle \vt,E^{\bm\xi}_j\rangle-\sum_{i=1}^N\langle\ut_i^{\bm\xi},E^{\bm\xi}_j\rangle\xi'_i+\sum_{i=1}^N\langle\wt,E^{\bm\xi}_{ji}\rangle\xi'_i=0,
	\qquad j=1,\dots,N,
\end{equation}
where we introduced the notation $E^{\bm\xi}_{ji}:=\partial_i E^{\bm\xi}_j$.
Rewrite \eqref{eq:xi} in the compact form
\begin{equation}\label{eq:xi-compact}
	D(\bm\xi,\wt)\bm\xi'=Y(\bm\xi,\vt),  
\end{equation}
where
\begin{equation*}
	D_{ji}(\bm\xi,\wt):=\langle\ut_i^{\bm\xi},E^{\bm\xi}_j\rangle-\langle\wt,E^{\bm\xi}_{ji}\rangle, \qquad \mbox{ and } \qquad Y_j(\bm\xi,\vt):=\langle \vt,E^{\bm\xi}_j\rangle.
\end{equation*}
Therefore, combining \eqref{eq:w-v} and \eqref{eq:xi-compact} we obtain the ODE-PDE coupled system
\begin{equation}\label{eq:w-v-xi}
	\begin{cases}
	\wt_t=\vt-\displaystyle\sum_{j=1}^N\ut_j^{\bm\xi}\xi'_j,\\
	\tau\vt_t=\mathcal{L}(\ut^{\bm\xi})+L^{\bm\xi}\wt+{(f_2\wt_x^2)}_x-\vt,\\
	D(\bm\xi,\wt)\bm\xi'=Y(\bm\xi,\vt).
	\end{cases}
\end{equation}
Now, let us define the slow channel where we will study the dynamics of \eqref{eq:w-v-xi}.
Let $\bm\xi$ such that $\bm h=\mathbf{G}(\bm\xi)\in\Omega_\rho$ and $\wt\in C^2(0,1)$ with $\wt=0$ at $x=0,1$; define
\begin{align*}
	A^{\bm\xi}(\wt)&:=-\langle L^{\bm\xi}\wt,\wt\rangle=\int_0^1\left[\e^2\wt_{xx}^2+F''(u^{\bm\xi})\wt^2_x\right]\,dx,\\
	B(\wt)&:=\int_0^1\left[\e^2\wt_{xx}^2+\wt^2_x\right]\,dx.
\end{align*}
We recall the following lemma of \cite{Bates-Xun1}.
\begin{lem}
For any $\wt\in C^2(0,1)$ with $\wt=0$ at $x=0,1$, we have
\begin{align}
	\|\wt\|_{{}_{L^\infty}}^2&\leq B(\wt), \label{eq:wt-B}\\
	\e\|\wt_x\|_{{}_{L^\infty}}^2&\leq(1+\e)B(\wt). \label{eq:wt_x-B}
\end{align}
Moreover, assume that $F$ satisfies \eqref{eq:ass-F}.
There exists $\rho_0>0$ such that if $\rho\in(0,\rho_0)$ and $\bm h=\mathbf{G}(\bm\xi)\in\Omega_\rho$, then for any $\wt$ as above and satisfying the orthogonality condition \eqref{eq:orthogonality}, we have
\begin{equation}\label{eq:A^xi-B}
	CA^{\bm\xi}(\wt)\geq\e^2B(\wt),
\end{equation}
for some positive constant $C$ independent of $\e$ and $\wt$.
\end{lem}
For the proof of this lemma see \cite[Lemmas 4.1 and 4.2]{Bates-Xun1}.
Let us define the energy functional
\begin{equation*}
	E^{\bm\xi}[\wt,\vt]:=\frac12A^{\bm\xi}(\wt)+\frac\tau2\|\vt\|^2+\e^\theta\tau\langle\wt,\vt\rangle, \qquad \qquad \mbox{ for } \, \theta>0,
\end{equation*}
and the slow channel
\begin{align*}
	\mathcal{Z}_{{}_{\Gamma,\rho}}:=\biggl\{(\ut,\vt)\,:\,\ut=\ut^{\bm\xi}+\wt,\;\; (\wt,\vt)\in H\times H^2(0,1),  \; \bm\xi \, &\mbox{ is such that }\, \bm h=\mathbf{G}(\bm\xi)\in\overline\Omega_\rho,\\
	& \qquad \mbox{and } E^{\bm{\xi}}[\wt,\vt]\leq\Gamma\e^{-2}\Psi(\bm h)\biggr\},
\end{align*}
for $\Gamma,\rho>0$, where the space $H$ and the barrier function $\Psi$ are defined in \eqref{eq:H} and \eqref{eq:barrier}, respectively.
Studying the dynamics of the solutions to \eqref{eq:integrated-CH-sys} in the slow channel $\mathcal{Z}_{{}_{\Gamma,\rho}}$ is equivalent to study the dynamics of the solutions to \eqref{eq:w-v-xi} in the set 
\begin{align*}
	\hat{\mathcal{Z}}_{{}_{\Gamma,\rho}}:=\biggl\{(\wt,\vt,\bm\xi) \, : \, (\wt,\vt)\in H\times H^2(0,1), \; \bm\xi \, &\mbox{ is such that }\,  \bm h=\mathbf{G}(\bm\xi)\in\overline\Omega_\rho,\\
	 &\qquad  \mbox{ and } E^{\bm{\xi}}[\wt,\vt]\leq\Gamma\e^{-2}\Psi(\bm h)\biggr\}.
\end{align*}
Hence, we will study the dynamics of \eqref{eq:w-v-xi} in the set $\hat{\mathcal{Z}}_{{}_{\Gamma,\rho}}$.
The first step is the following proposition, which gives estimates for solutions $(\wt,\vt,\bm\xi)\in\hat{\mathcal{Z}}_{{}_{\Gamma,\rho}}$ to \eqref{eq:w-v-xi}.
\begin{prop}\label{prop:estimate-channel}
Let us assume that $F\in C^4(\mathbb{R})$ satisfies conditions \eqref{eq:ass-F}.
Given $N\in\mathbb{N}$ and $\delta\in(0,1/(N+1))$, there exists $\e_0>0$ such that if $\e,\rho$ satisfy \eqref{eq:triangle}, $\theta>1$,
and $(\wt,\vt,\bm\xi)\in\hat{\mathcal{Z}}_{{}_{\Gamma,\rho}}$, then
\begin{align}
	A^{\bm\xi}(\wt)\leq C E^{\bm\xi}[\wt,\vt]\leq C\Gamma\e^{-2}\exp\left(-\frac{2Al^{\bm h}}{\e}\right), \label{eq:A^xi-E}\\
	\e^2\|\wt\|^2_{{}_{L^\infty}}+\tau\|\vt\|^2_{{}_{L^2}}\leq CE^{\bm\xi}[\wt,\vt]\leq C\Gamma\e^{-2}\exp\left(-\frac{2Al^{\bm h}}{\e}\right), \label{eq:wt-vt-E}
\end{align}
for some positive constant $C>0$ (independent of $\e$, $\tau$ and $\theta$).

Moreover, if $(\wt,\vt,\bm\xi)\in\hat{\mathcal{Z}}_{{}_{\Gamma,\rho}}$ is a solution to \eqref{eq:w-v-xi} for $t\in[0,T]$, then
\begin{equation}\label{eq:xi'-estimate}
	|\bm\xi'|_{{}_\infty}\leq C\e^{-2}\tau^{-1/2}\exp\left(-\frac{Al^{\bm h}}{\e}\right).
\end{equation}
\end{prop}
\begin{proof}
Let us prove \eqref{eq:A^xi-E}. 
Using Young inequality, we infer
\begin{equation*}
	2\e^\theta|\langle\wt,\vt\rangle|\leq\e^{2\theta}\|\wt\|^2_{{}_{L^{\infty}}}+\|\vt\|^2,
\end{equation*}
and so, for the definitions of $A^{\bm\xi}$ and $E^{\bm\xi}$, we have
\begin{equation*}
	A^{\bm\xi}(\wt)=2E^{\bm\xi}[\wt,\vt]-\tau\|\vt\|^2-2\e^\theta\tau\langle\wt,\vt\rangle\leq 2E^{\bm\xi}[\wt,\vt]+\e^{2\theta}\tau\|\wt\|^2_{{}_{L^{\infty}}}.
\end{equation*}
Using \eqref{eq:wt-B} and \eqref{eq:A^xi-B}, we deduce 
\begin{equation*}
	\|\wt\|^2_{{}_{L^{\infty}}}\leq B(\wt)\leq C\e^{-2} A^{\bm\xi}(\wt),
\end{equation*}
and then
\begin{equation*}
	A^{\bm\xi}(\wt)\leq 2E^{\bm\xi}[\wt,\vt]+C\e^{2(\theta-1)}\tau A^{\bm\xi}(\wt).
\end{equation*}
Since $\theta>1$, we can choose $\e_0$ so small that $C\e^{2(\theta-1)}\tau\leq\nu<1$ for $\e\in(0,\e_0)$, and conclude that
\begin{equation*}
	A^{\bm\xi}(\wt)\leq CE^{\bm\xi}[\wt,\vt].
\end{equation*}
The second inequality of \eqref{eq:A^xi-E} follows from the facts that $E^{\bm\xi}[\wt,\vt]\leq\Gamma\e^{-2}\Psi(\bm h)$ in $\hat{\mathcal{Z}}_{{}_{\Gamma,\rho}}$ and that
the definition of the barrier function $\Psi$ \eqref{eq:barrier} and \eqref{eq:alfaj} imply that
\begin{equation}\label{eq:Psi<exp}
	\Psi(\bm h)\leq C\exp\left(-\frac{2Al^{\bm h}}{\e}\right),
\end{equation}
for some $C>0$ independent of $\e$.
The proof of \eqref{eq:wt-vt-E} is similar.
From \eqref{eq:A^xi-B} and Young inequality, one has
\begin{equation*}
	E^{\bm\xi}[\wt,\vt]\geq C\e^2B(\wt)+\frac\tau2\|\vt\|^2-\e^{2\theta}\tau\|\wt\|^2-\frac\tau4\|\vt\|^2.
\end{equation*}
Hence, for \eqref{eq:wt-B} we get
\begin{equation*}
	E^{\bm\xi}[\wt,\vt]\geq C\e^2\|\wt\|^2_{{}_{L^\infty}}+\frac\tau4\|\vt\|^2-\e^{2\theta}\tau\|\wt\|^2_{{}_{L^\infty}}\geq \left(C-\e^{2(\theta-1)}\tau\right)\e^2\|\wt\|^2_{{}_{L^\infty}}+\frac\tau4\|\vt\|^2,
\end{equation*}
and we obtain \eqref{eq:wt-vt-E} choosing $\e$ sufficiently small (again since $\theta>1$) and using \eqref{eq:Psi<exp}.

It remains to prove \eqref{eq:xi'-estimate}.
Let us consider the equation for $\bm\xi$ in \eqref{eq:w-v-xi} and the matrix $D(\bm\xi,\wt)$ of elements $D_{ij}(\bm\xi,\wt):=\langle\ut_j^{\bm\xi},E^{\bm\xi}_i\rangle-\langle\wt,E^{\bm\xi}_{ij}\rangle$.
These elements have been already studied in \cite{Bates-Xun1,Bates-Xun2} and one has
\begin{equation}\label{eq:aij}
	a_{ij}:=\langle\ut_j^{\bm\xi},E_i\rangle=\begin{cases}
	(-1)^{i+j}4l_{j+1}+\mathcal{O}(\e), \qquad \qquad & i\geq j,	\\
	\mathcal{O}(\e), & i\leq j,
	\end{cases}
\end{equation}
where $l_j:=h_j-h_{j-1}$ is the distance between the layers (see formulas (4.27) in \cite{Bates-Xun2}), and $\|E^{\bm\xi}_{ij}\|\leq C\e^{-1/2}$.
These results can be obtained by using the formulas \eqref{eq:ut-j-xi} and \eqref{eq:Ej} for $\ut_j^{\bm\xi}$ and $E^{\bm\xi}_j$.
In particular, the bound for $\|E^{\bm\xi}_{ij}\|$ can be easily obtained by differentiating with respect to $\xi_j$ the formula \eqref{eq:Ej} 
without the exponentially small terms and by using \eqref{eq:uxi_j}.
From \eqref{eq:wt-vt-E}, it follows that
\begin{equation*}
	|\langle\wt,E^{\bm\xi}_{ji}\rangle|\leq\|\wt\|\|E^{\bm\xi}_{ji}\|\leq C\e^{-5/2}\exp\left(-\frac{Al^{\bm h}}{\e}\right).
\end{equation*}
Hence, for $\e$ sufficiently small, we have 
\begin{equation*}
	D(\bm\xi,\wt):=\left(\begin{array}{ccccc} 4l_2 & 0 & 0 & \dots & 0\\
	-4l_3 & 4l_3 & 0 & \dots & 0\\
	4l_4 & -4l_4 & 4l_4 & \dots & 0\\
	\dots & \dots & \dots & \dots & \dots\\
	(-1)^{N-1}4l_{N+1} & (-1)^{N-2}4l_{N+1} & (-1)^{N-3}4l_{N+1} & \dots & 4l_{N+1}
	\end{array}\right)+\mathcal{O}(\e),
\end{equation*}
and its inverse
\begin{equation*}
	D^{-1}(\bm\xi,\wt)=\left(\begin{array}{cccccc} \displaystyle{\frac1{4l_2}} & 0 & 0 & \dots & 0 & 0\\
	\displaystyle{\frac1{4l_2}} & \displaystyle{\frac1{4l_3}} & 0 & \dots & 0 & 0\\
	0 & \displaystyle\frac1{4l_3} & \displaystyle\frac1{4l_4} & \dots & 0 & 0\\
	\dots & \dots & \dots & \dots & \dots\\
	0 & 0 & 0 & \dots &  \displaystyle\frac1{4l_N} &  \displaystyle\frac1{4l_{N+1}}\\
	\end{array}\right)+\mathcal{O}(\e).
\end{equation*}
Let us rewrite the equation for $\bm\xi$ in \eqref{eq:w-v-xi} as 
\begin{equation*}
	\bm\xi'=D^{-1}(\bm\xi,\wt)Y(\bm\xi,\vt).
\end{equation*}
Since
\begin{equation*}
	|Y_j(\bm\xi,\vt)|=|\langle \vt,E^{\bm\xi}_j\rangle|\leq\|\vt\|\|E^{\bm\xi}_j\|\leq C\|\vt\|,
\end{equation*}
where in the last passage we used the formula \eqref{eq:Ej} for $E^{\bm\xi}_j$,
we deduce
\begin{equation*}
	|\bm\xi'|_{{}_{\infty}}\leq C\|D^{-1}(\bm\xi,\wt)\|_{{}_{\infty}}\|\vt\|,
\end{equation*}
where $\|\cdot\|_{{}_{\infty}}$ denotes the matrix norm induced by the vector norm $|\cdot|_{{}_{\infty}}$.
To estimate such matrix norm, we use the assumption $\bm h\in\Omega_\rho$, which implies $l_j>\e/\rho$ for any $j$.
Therefore, we can conclude that
\begin{equation}\label{eq:xi-vt}
	|\bm\xi'|_{{}_{\infty}}\leq C\e^{-1}\|\vt\|,
\end{equation}
and the proof of \eqref{eq:xi'-estimate} follows from \eqref{eq:wt-vt-E}.
\end{proof}

\subsection{Main result}
Thanks to the estimates \eqref{eq:A^xi-E}, \eqref{eq:wt-vt-E} and \eqref{eq:xi'-estimate}, we can state that if $(\wt,\vt,\bm\xi)\in\hat{\mathcal{Z}}_{{}_{\Gamma,\rho}}$ is a solution to \eqref{eq:w-v-xi} in $[0,T]$,
then the $L^\infty$--norm of $\wt$, the $L^2$--norm of $\vt$ and the velocity of $\bm\xi$ are exponentially small in $\e$.
This implies that if $u=u^{\bm h}+w$ is a solution to \eqref{eq:hyp-CH} such that $(\ut,\ut_t)\in\mathcal{Z}_{{}_{\Gamma,\rho}}$ for $t\in[0,T]$,
then the $L^\infty$--norm of $w$ and the velocity of the transition points $(h_1,\dots,h_{N+1})$ are exponentially small.
Indeed, to estimate the norm of $w$, we use \eqref{eq:wt_x-B}, \eqref{eq:A^xi-B} and \eqref{eq:A^xi-E} and we get
\begin{equation*}
	\|w(\cdot,t)\|_{{}_{L^\infty}}=\|\wt_x(\cdot,t)\|_{{}_{L^\infty}}\leq C\e^{-1/2} B(\wt)^{1/2}\leq C\e^{-3/2} A^{\bm\xi}(\wt)^{1/2}\leq C\e^{-5/2}\exp\left(-\frac{Al^{\bm h}}{\e}\right),
\end{equation*}
for some $C>0$ independent of $\e$.
On the other hand, the velocity of the transition points $(h_1,\dots,h_N)$ is exponentially small for \eqref{eq:xi'-estimate} and the fact that $h_i=\xi_i$ for $i=1,\dots,N$;
to estimate the velocity of $h_{N+1}$, we use \eqref{eq:N+1-der} and the relation
\begin{equation}\label{eq:h'_N+1}
	h'_{N+1}=\sum_{j=1}^N\frac{\partial h_{N+1}}{\partial h_j}h'_j=\sum_{j=1}^N\left[(-1)^{N-j}+\mathcal{O}\left(\e^{-1}\exp\left(-\frac{Al^{\bm h}}\e\right)\right)\right]h'_j.
\end{equation}
From \eqref{eq:h'_N+1} and \eqref{eq:xi'-estimate} it follows that
\begin{equation*}
	|h'_{N+1}(t)|\leq C\e^{-2}\tau^{-1/2}\exp\left(-\frac{Al^{\bm h}}{\e}\right).
\end{equation*}
Therefore, we can state that if $u=u^{\bm h}+w$ is a solution to \eqref{eq:hyp-CH} such that $(\ut,\ut_t)\in\mathcal{Z}_{{}_{\Gamma,\rho}}$ for $t\in[0,T]$, then
\begin{equation*}
	\|u-u^{\bm h}\|_{{}_{L^\infty}}\leq C\e^{-5/2}\exp\left(-\frac{Al^{\bm h}}{\e}\right), \qquad \qquad |\bm h'|_{{}_\infty}\leq C\e^{-2}\tau^{-1/2}\exp\left(-\frac{Al^{\bm h}}{\e}\right),
\end{equation*}
for $t\in[0,T]$.
In other words, there exists a neighborhood of the manifold $\mathcal{M}_{{}_0}$ where the solution $u$ to \eqref{eq:hyp-CH}-\eqref{eq:Neumann}
is well approximated by $u^{\bm h}$; thus, $u$ is a function with $N+1$ transitions between $-1$ and $+1$,
and the velocity of the transition points is exponentially small.
Let us focus the attention on a lower bound of the time $T_\e$ taken for the solution to leave such neighborhood of $\mathcal{M}_{{}_0}$. 
To this aim, we observe that a solution can leave the slow channel $\mathcal{Z}_{{}_{\Gamma,\rho}}$ either if $\bm h=\mathbf{G}(\bm\xi)\in\partial\Omega_\rho$,
meaning that two transition points are close enough, namely $h_j-h_{j-1}=\e/\rho$ for some $j$, 
or if the energy functional is large enough, precisely $E^{\bm{\xi}}[\wt,\vt]=\Gamma\e^{-2}\Psi(\bm h)$.
We will prove that solutions leave the slow channel only if two transition points are close enough;
then, since the transition points move with exponentially small velocity, the time taken for the solution to leave the slow channel is exponentially large.
Precisely, we will prove the following result.
\begin{thm}\label{thm:main}
Let us assume that $F\in C^4(\mathbb{R})$ satisfies conditions \eqref{eq:ass-F} and consider the IBVP \eqref{eq:hyp-CH}-\eqref{eq:Neumann}-\eqref{eq:initial} with $u_1$ satisfying \eqref{eq:ass-u1}.
Given $N\in\mathbb{N}$ and $\delta\in(0,1/(N+1))$, there exist $\e_0,\theta_0>0$ and $\Gamma_2>\Gamma_1>0$ such that if $\e,\rho$ satisfy \eqref{eq:triangle}, $\theta>\theta_0$,
$\Gamma\in[\Gamma_1,\Gamma_2]$, and the initial datum $(u_0,u_1)$ is such that
\begin{equation*}
	(\ut_0,\ut_1)\in\,\stackrel{\circ}{\mathcal{Z}}_{{}_{\Gamma,\rho}}=\bigl\{(\ut,\vt)\in\mathcal{Z}_{{}_{\Gamma,\rho}}\, : \, {\bm h}=\mathbf{G}(\bm\xi)\in\Omega_\rho
	\;\;\textrm{and}\;\; E^{\bm\xi}[w,v]<\Gamma\e^{-2}\Psi({\bm h})\bigr\},
\end{equation*}
then the solution $(u,u_t)$ is such that $(\ut,\ut_t)$ remains in $\mathcal{Z}_{{}_{\Gamma,\rho}}$ for a time $T_\varepsilon>0$, and for any $t\in[0,T_\varepsilon]$ one has
\begin{equation}\label{eq:u-uh,h'}
	\|u-u^{\bm h}\|_{{}_{L^\infty}}\leq C\e^{-5/2}\exp\left(-\frac{Al^{\bm h}}{\e}\right), \qquad \qquad |\bm h'|_{{}_\infty}\leq C\e^{-2}\tau^{-1/2}\exp\left(-\frac{Al^{\bm h}}{\e}\right),
\end{equation}
where $A:=\sqrt{\min\{F''(-1),F''(+1)\}}$, $\ell^{\bm h}:=\min\{h_j-h_{j-1}\}$ and $|\cdot|_{{}_{\infty}}$ denotes the maximum norm in $\mathbb{R}^N$.
Moreover, there exists $C>0$ such that
\begin{equation*}
		T_\varepsilon\geq C\e^2\tau^{1/2}(\ell^{\bm h(0)}-\varepsilon/\rho)\exp(A\delta /\varepsilon).
\end{equation*}
\end{thm}
The proof of Theorem \ref{thm:main} is based on the following proposition, which gives an estimate on the time derivative of $E^{\bm\xi}[\wt,\vt]$ along the solutions to the system \eqref{eq:w-v-xi}.
\begin{prop}\label{prop:energy-estimate}
Let us assume that $F\in C^4(\mathbb{R})$ satisfies conditions \eqref{eq:ass-F}.
Given $N\in\mathbb{N}$ and $\delta\in(0,1/(N+1))$, there exist $\e_0,\theta_0>0$ and $\Gamma_2>\Gamma_1>0$ such that if $\e,\rho$ satisfy \eqref{eq:triangle}, $\theta>\theta_0$,
$\Gamma\in[\Gamma_1,\Gamma_2]$, and $(\wt,\vt,\bm\xi)\in\hat{\mathcal{Z}}_{{}_{\Gamma,\rho}}$ is a solution to \eqref{eq:w-v-xi} for $t\in[0,T]$, then for some $\eta\in(0,1)$ and $\mu>0$, we have
\begin{equation}\label{eq:E-Psi}
	\frac d{dt}\bigl\{E^{\bm{\xi}}[\wt,\vt]-\Gamma\e^{-2}\Psi(\bm h)\bigr\} \leq
	-\eta\e^\mu\bigl\{E^{\bm{\xi}}[\wt,\vt]-\Gamma\e^{-2}\Psi(\bm h)\bigr\}, \qquad \qquad \mbox{for } t\in[0,T].
\end{equation}
\end{prop}
\begin{proof}
In all the proof, symbols $C, c, \eta$ denote generic positive constants, independent on $\varepsilon$, and with $\eta\in(0,1)$.
Let us differentiate with respect to $t$ the three terms of the energy functional $E^{\bm\xi}$.
For the first term, direct differentiation and the first equation of \eqref{eq:w-v-xi} give
\begin{equation*}
	\begin{aligned}
	\frac{d}{dt}\left\{\frac12 A^{\bm\xi}(\wt)\right\} & =-\frac{d}{dt}\left\{\frac12\langle L^{\bm\xi}\wt,\wt\rangle\right\}
	=-\langle L^{\bm\xi}\wt,\wt_t\rangle+\frac12\langle{(F''(u^{\bm \xi}))}_t,\wt_x^2\rangle\\
	& =-\langle L^{\bm\xi}\wt,\vt\rangle+\sum_{j=1}^N\xi'_j\langle L^{\bm\xi}\wt,\ut_j^{\bm\xi}\rangle
		+\frac12\sum_{j=1}^N\xi'_j\langle F'''(u^{\bm \xi})\ut_j^{\bm\xi},\wt_x^2\rangle.
	\end{aligned}
\end{equation*}
Using the self-adjointness of the operator $L^{\bm\xi}$ and inequality \eqref{eq:xi-vt}, we infer
\begin{equation*}
	\sum_{j=1}^N|\xi'_j\langle L^{\bm\xi}\wt,\ut_j^{\bm\xi}\rangle|= \sum_{j=1}^N|\xi'_j\langle \wt,L^{\bm\xi}\ut_j^{\bm\xi}\rangle| \leq C\e^{-1}\|\vt\|\|\wt\|\max_j\|L^{\bm\xi}\ut_j^{\bm\xi}\|.
%	& \leq\eta\|\vt\|^2+C\e^{-2}\|\wt\|_{{}_{\infty}}^2\max_j\|L^{\bm\xi}\ut_j^{\bm\xi}\|^2	\\
%	& \leq \eta\|\vt\|^2+C\e^{-4}A^{\bm\xi}(\wt)\max_j\|L^{\bm\xi}\ut_j^{\bm\xi}\|^2.
\end{equation*}
For the last term of the latter inequality, we have that
\begin{align*}
	L^{\bm\xi}\ut_j^{\bm\xi} & =L^{\bm\xi}\ut_j^{\bm h}+\frac{\partial h_{N+1}}{\partial h_j}L^{\bm\xi}\ut_{N+1}^{\bm h}
	=\frac{\partial}{\partial h_j}\mathcal{L}(\ut^{\bm h})+\frac{\partial h_{N+1}}{\partial h_j}\frac{\partial}{\partial h_{N+1}}\mathcal{L}(\ut^{\bm h})\\
	& = -\frac{\partial}{\partial h_j}\frac{\partial}{\partial x}\mathcal{L}^{AC}(u^{\bm h})-\frac{\partial h_{N+1}}{\partial h_j}\frac{\partial}{\partial h_{N+1}}\frac{\partial}{\partial x}\mathcal{L}^{AC}(u^{\bm h}),
\end{align*}
and from  \cite[Lemma 5.2]{Bates-Xun1}, it follows that
\begin{equation*}
	\|L^{\bm\xi}\ut_j^{\bm\xi}\|\leq C\e^{-4}\exp\left(-\frac{Al^{\bm h}}{2\e}\right)\leq C\e^{-4}\exp\left(-\frac{A\delta}{2\e}\right),
\end{equation*}
where we used \eqref{eq:triangle}.
On the other hand, the formula \eqref{eq:ut-j-xi} and the inequalities \eqref{eq:wt_x-B}, \eqref{eq:A^xi-B} and \eqref{eq:xi-vt} yield  
\begin{align*}
	\left|\sum_{j=1}^N\xi'_j\langle F'''(u^{\bm \xi})\ut_j^{\bm\xi},\wt_x^2\rangle\right| & \leq C|\bm\xi|_{{}_{\infty}}\|\wt_x\|_{{}_{L^\infty}}\|\wt_x\|^2\max_j\|\ut_j^{\bm\xi}\|\leq C\e^{-2}B(\wt)\|\vt\|\\
	& \leq C\e^{-4}A^{\bm\xi}(\wt)\|\vt\|.
\end{align*}
Therefore, for the first term of the energy we conclude
\begin{equation}\label{eq:first-energy}
	\frac{d}{dt}\left\{\frac12 A^{\bm\xi}(\wt)\right\}\leq-\langle L^{\bm\xi}\wt,\vt\rangle+C\e^{-5}\exp(-c/\e)\|\vt\|\|\wt\|+C\e^{-4}A^{\bm\xi}(\wt)\|\vt\|.
\end{equation}
For what concerns the second term in the energy $E^{\bm\xi}$, the second equation of \eqref{eq:w-v-xi} gives
\begin{equation*}
	\begin{aligned}
	\frac{d}{dt}\Bigl\{\frac\tau2\|\vt\|^2\Bigr\}& =\langle \tau v_t,v\rangle
		=\langle\mathcal{L}(\ut^{\bm\xi})+L^{\bm\xi}\wt+{(f_2\wt^2_x)}_x-\vt,\vt\rangle\\
	&\leq\langle L^{\bm\xi}\wt,\vt\rangle+\|\mathcal{L}(\ut^{\bm\xi})\|\|v\|+\langle{(f_2\wt^2_x)}_x,\vt\rangle-\|\vt\|^2\\
	&\leq \langle L^{\bm\xi}\wt,\vt\rangle-\frac12\|\vt\|^2+C\|\mathcal{L}(\ut^{\bm\xi})\|^2+\langle{(f_2\wt^2_x)}_x,\vt\rangle.
	\end{aligned}
\end{equation*}
By expanding
\begin{align*}
	{(f_2\wt^2_x)}_x & = {(f_2)}_x\wt^2_x+2f_2\wt_x\wt_{xx}\\
	& = \wt^2_x\int_0^1(1-s)F''''(\ut_x^{\bm\xi}+s\wt_x)(\ut^{\bm\xi}_{xx}+s\wt_{xx})\,ds+2f_2\wt_x\wt_{xx},
\end{align*}
we deduce the estimate
\begin{equation}\label{eq:f2wx}
	\begin{aligned}
	|\langle{(f_2\wt^2_x)}_x,\vt\rangle| & \leq\|{(f_2\wt^2_x)}_x\|\|\vt\|\leq C\left(\|\wt_x\|^2_{{}_{L^\infty}}\|\ut^{\bm\xi}_{xx}+\wt_{xx}\|+\|\wt_x\|_{{}_{L^\infty}}\|\wt_{xx}\|\right)\|\vt\|\\
	& \leq C\left\{\e^{-1}B(\wt)\left(\e^{-1}+\e^{-1}B(\wt)^{1/2}\right)+\e^{-3/2}B(\wt)\right\}\|\vt\|,\\
	& \leq C\left\{\e^{-4}+\e^{-5}A^{\bm\xi}(\wt)^{1/2}\right\}\|\vt\|A^{\bm\xi}(\wt),
	\end{aligned}
\end{equation}
where the estimates \eqref{eq:wt_x-B}, \eqref{eq:A^xi-B},
\begin{equation*}
	\|\ut^{\bm\xi}_{xx}\|=\|u^{\bm\xi}_x\|\leq C\e^{-1}, \qquad \mbox{ and } \qquad \|\wt_{xx}\|^2\leq\e^{-2}B(\wt).
\end{equation*}
have been used.
Hence, we obtain
\begin{equation}\label{eq:second-energy}
	\frac{d}{dt}\Bigl\{\frac\tau2\|\vt\|^2\Bigr\}\leq\langle L^{\bm\xi}\wt,\vt\rangle-\frac12\|\vt\|^2+C\|\mathcal{L}(\ut^{\bm\xi})\|^2+C\left\{1+\e^{-1}A^{\bm\xi}(\wt)^{1/2}\right\}\e^{-4}\|\vt\|A^{\bm\xi}(\wt).
\end{equation}
Finally, the time derivative of the scalar product $\langle w,\tau v\rangle$ can be bounded by
\begin{equation*}
	\begin{aligned}
	\frac{d}{dt}\langle \wt,\tau\vt\rangle %& =\langle w_t,\tau v\rangle+\langle w,\tau v_t\rangle\\
	&=\tau\|\vt\|^2-\tau\sum_{j=1}^N\xi'_j\langle\ut_j^{\bm\xi},\vt\rangle+\langle\wt,\mathcal{L}(\ut^{\bm\xi})+L^{\bm\xi}\wt+{(f_2\wt_x^2)}_x-\vt\rangle\\
	&\leq\tau\|\vt\|^2+\tau|\bm\xi'|_{{}_{\infty}}\|\ut_j^{\bm\xi}\|\|\vt\|+\|\wt\|\|\mathcal{L}(\ut^{\bm\xi})\|-A^{\bm\xi}(\wt)-\langle\wt,\vt\rangle+\langle\wt,{(f_2\wt_x^2)}_x\rangle.
	\end{aligned}
\end{equation*}
where we used that $A^{\bm\xi}(\wt)=-\langle\wt,L^{\bm\xi}\wt\rangle$.
By using \eqref{eq:ut-j-xi}, \eqref{eq:xi-vt} and estimating as in \eqref{eq:f2wx}, we infer
\begin{equation*}
	\begin{aligned}
		\e^\theta\frac{d}{dt}\langle \wt,\tau\vt\rangle\leq-\e^\theta A^{\bm\xi}(\wt)+\e^\theta\tau(1+C\e^{-1})&\|\vt\|^2-\frac{\e^\theta}\tau\langle \wt,\tau\vt\rangle+\e^\theta\|\wt\|\|\mathcal{L}(\ut^{\bm\xi})\|\\
		&+C\left\{1+\e^{-1}A^{\bm\xi}(\wt)^{1/2}\right\}\e^{\theta-4}\|\wt\|A^{\bm\xi}(\wt).
	\end{aligned}
\end{equation*}
For Young inequality, we have
\begin{equation*}
	\e^\theta\|\wt\|\|\mathcal{L}(\ut^{\bm\xi})\|\leq C\e^{2\theta}\|\wt\|^2+C\|\mathcal{L}(\ut^{\bm\xi})\|^2\leq C\e^{2(\theta-1)}A^{\bm\xi}(\wt)+C\|\mathcal{L}(\ut^{\bm\xi})\|^2,
\end{equation*}
and we can estimate the third term of $E^{\bm\xi}$ as
\begin{equation}\label{eq:third-energy}
	\begin{aligned}
		\e^\theta\frac{d}{dt}\langle \wt,\tau\vt\rangle\leq&-\e^\theta A^{\bm\xi}(\wt)-\frac{\e^\theta}\tau\langle \wt,\tau\vt\rangle+\e^\theta\tau(1+C\e^{-1})\|\vt\|^2+C\|\mathcal{L}(\ut^{\bm\xi})\|^2\\
		&+C\e^{2(\theta-1)}A^{\bm\xi}(\wt)+C\left\{1+\e^{-1}A^{\bm\xi}(\wt)^{1/2}\right\}\e^{\theta-4}\|\wt\|A^{\bm\xi}(\wt).
	\end{aligned}
\end{equation}
Collecting \eqref{eq:first-energy}, \eqref{eq:second-energy} and \eqref{eq:third-energy}, we deduce 
\begin{equation*}
	\begin{aligned}
	\frac{d}{dt}E^{\bm\xi}[\wt,\vt] \leq &-\e^\theta A^{\bm\xi}(\wt)-\left(\frac12-\e^\theta\tau(1+C\e^{-1})\right)\|\vt\|^2-\frac{\e^\theta}\tau\langle \wt,\tau\vt\rangle+C\|\mathcal{L}(\ut^{\bm\xi})\|^2\\
	&+C\e^{2(\theta-1)}A^{\bm\xi}(\wt)+R^{\bm\xi}[\wt,\vt],
	\end{aligned}
\end{equation*}
where
\begin{align*}
	R^{\bm\xi}[\wt,\vt]:= C\e^{-5}\exp(-c/\e)\|\vt\|&\|\wt\|+C\e^{-4}A^{\bm\xi}(\wt)\|\vt\|\\
	&+C\e^{-4}\left\{1+\e^{-1}A^{\bm\xi}(\wt)^{1/2}\right\}\left\{\|\vt\|+\e^{\theta}\|\wt\|\right\}A^{\bm\xi}(\wt).
\end{align*}
Choosing $\theta>2$ and $\e_0$ so small that
\begin{equation}\label{eq:cond-eps-tau}
	C\e^{\theta-1}\tau\leq\frac12-\eta,
\end{equation}
for any $\e\in(0,\e_0)$, we obtain
\begin{equation*}
	\frac{d}{dt}E^{\bm\xi}[\wt,\vt] \leq -\eta\e^\theta A^{\bm\xi}(\wt)-\eta\|\vt\|^2-\frac{\e^\theta}\tau\langle \wt,\tau\vt\rangle+C\|\mathcal{L}(\ut^{\bm\xi})\|^2+R^{\bm\xi}[\wt,\vt].
\end{equation*}
Therefore, we conclude that there exists $\mu>0$ (independent on $\e$) such that
\begin{equation*}
	\frac{d}{dt}E^{\bm\xi}[\wt,\vt] \leq -\eta\e^\mu E^{\bm\xi}[\wt,\vt]-\eta\|\vt\|^2+C\|\mathcal{L}(\ut^{\bm\xi})\|^2+R^{\bm\xi}[\wt,\vt],
\end{equation*}
for some $\eta\in(0,1)$. 
Indeed, the condition \eqref{eq:cond-eps-tau} implies $\e^\theta/\tau>C\e^{2\theta-1}$ and, since $\theta>2$, we can choose $\mu\geq2\theta-1$.

Now, let us use that $(\wt,\vt,\bm\xi)\in\hat{\mathcal{Z}}_{{}_{\Gamma,\rho}}$ for $t\in[0,T]$;
from Proposition \ref{prop:estimate-channel} it follows that there exists $\e_0$ (dependent on $\Gamma$ and $\tau$) such that
\begin{equation*}
	R^{\bm\xi}[\wt,\vt]\leq C\exp(-c/\e)\Gamma\e^{-2}\Psi(\bm h),
\end{equation*}
for any $\e\in(0,\e_0)$.
Since, for \eqref{eq:barrier} and \eqref{eq:Lut^xi} one has
\begin{equation*}
	\|\mathcal{L}(\ut^{\bm\xi})\|^2\leq C\varepsilon^{-2}\Psi(\bm h),
	%\leq C\e^{-2}\exp(-2A\ell^{\bm h}/\varepsilon),
\end{equation*}
we infer
\begin{equation}\label{eq:d/dtE}
	\frac{d}{dt}E^{\bm\xi}[\wt,\vt]\leq -\eta\e^\mu E^{\bm\xi}[\wt,\vt]-\eta\|\vt\|^2+C\e^{-2}\Psi(\bm h).
\end{equation}
Now, let us compute the time derivative of the barrier function $\Psi$.
Direct differentiation gives
\begin{equation*}
	\frac{d\Psi}{dt}=2\sum_{i,j=1}^{N+1}\langle\mathcal{L}^{AC}(u^{\bm h}),k^{\bm h}_j\rangle
		\Bigl\{\langle\mathcal{L}^{AC}(u^{\bm h}),k_{ji}^{\bm h}\rangle+\langle L^{AC}u^{\bm h}_i,k^{\bm h}_j\rangle\Bigr\}h'_i,
\end{equation*}
where $L^{AC}$ is the linearization of $\mathcal{L}^{AC}(u)$ about $u^{\bm h}$, i.e.
\begin{equation*}
	L^{AC}w:=\varepsilon^2w_{xx}-F''(u^{\bm h})w.
\end{equation*} 
Using the estimates provided by inequalities \eqref{eq:ineq-k} and \eqref{eq:xi-vt}, we deduce
\begin{equation*}
	\begin{aligned}
	\bigl|\langle\mathcal{L}^{AC}(u^{\bm h}),k_{ji}^{\bm h}\rangle h_i'\bigr|
	&\leq|\bm h'|_{{}_\infty}\|\mathcal{L}^{AC}(u^{\bm h})\|\|k^{\bm h}_{ji}\|
		\leq C\varepsilon^{-5/2}\|\mathcal{L}^{AC}(u^{\bm h})\|\|v\|,\\
	\bigl|\langle L^{AC}u^{\bm h}_i,k^{\bm h}_j\rangle h'_i\bigr|
	&\leq|\bm h'|_{{}_\infty}\|k^{\bm h}_j\|\|L^{AC}u^{\bm h}_i\|
		\leq C\exp(-c/\varepsilon)\|v\|,
	\end{aligned}
\end{equation*}
where in the last passage the inequality $\|L^{AC}u^{\bm h}_i\|\leq C\e^{-1/2}\exp(-Al^{\bm h}/\varepsilon)$ has been used (see \cite[Proposition 7.2]{Carr-Pego2}).
Thus, since
\begin{equation*}
	|\langle\mathcal{L}^{AC}(u^{\bm h}),k^{\bm h}_j\rangle|\leq C\varepsilon^{-1/2} \|\mathcal{L}^{AC}(u^{\bm h})\|,
\end{equation*}
we deduce the bound
\begin{equation*}
	\left|\frac{d\Psi}{dt}\right|
	\leq C\varepsilon^{-1/2} \left\{\varepsilon^{-2}\|\mathcal{L}^{AC}(u^{\bm h})\|+\exp(-c/\varepsilon)\right\}\|\mathcal{L}^{AC}(u^{\bm h})\|\|v\|.
\end{equation*}
It is well known (see \cite[Proposition 3.5]{Carr-Pego}) that
\begin{equation}
	\|\mathcal{L}^{AC}(u^{\bm h})\|^2\leq C\e\sum_{j=1}^{N+1}|\alpha^{j+1}-\alpha^j|^2\leq C\e\Psi(\bm h),
\end{equation}
where we used the definition of $\Psi$ \eqref{eq:barrier}.
Therefore, we obtain
\begin{equation*}
	\begin{aligned}
	\left|\Gamma\frac{d\Psi}{dt}\right|
	&\leq C\,\Gamma\bigl\{\e^{-3/2}\Psi^{1/2}+\exp(-c/\varepsilon)\bigr\}\|v\|\Psi^{1/2}\\
	&\leq \eta\|v\|^2+C\,\Gamma^2\bigl\{\varepsilon^{-3/2}\Psi^{1/2}+\exp(-c/\varepsilon)\bigr\}^2\Psi.
	\end{aligned}
\end{equation*}
Hence, observing that $\Psi\leq C\exp\bigl(-c/\varepsilon\bigr)$, we end up with
\begin{equation}\label{eq:Psi'}
	\left|\Gamma\frac{d\Psi}{dt}\right|\leq \eta\|v\|^2+C\,\Gamma^2\exp(-c/\varepsilon)\Psi.
\end{equation}
Combining \eqref{eq:d/dtE} and \eqref{eq:Psi'}, we obtain that if $(\bm\xi,\wt,\vt)\in\hat{\mathcal{Z}}_{{}_{\Gamma,\rho}}$ is a solution of \eqref{eq:w-v-xi}, then
\begin{equation*}
	\frac d{dt}\bigl\{E^{\bm{\xi}}[\wt,\vt]-\Gamma\e^{-2}\Psi(\bm h)\bigr\} \leq
	-\eta\e^\mu E^{\bm\xi}[\wt,\vt]+C\bigl(\varepsilon^{-2}+\Gamma^2\exp(-c/\varepsilon)\bigr)\Psi,
\end{equation*}
for some $\eta\in(0,1)$.
Therefore the estimate \eqref{eq:E-Psi} follows from 
\begin{equation*}
C\exp(-c/\varepsilon)\Gamma^2-\eta\,\varepsilon^{\mu-2}\Gamma +C\varepsilon^{-2}\leq 0,
\end{equation*}
 and the latter is verified for $\Gamma\in [\Gamma_1,\Gamma_2]$, provided $\varepsilon\in(0,\varepsilon_0)$ with $\varepsilon_0$ sufficiently small so that 
 $\eta^2\varepsilon^{2\mu} - 4C^2\e^{2}\exp(-c/\varepsilon) > 0$.
\end{proof}
\begin{rem}
Regarding the role of the parameter $\tau$ and its possible dependence on $\e$, we observe that Propositions \ref{prop:estimate-channel} and \ref{prop:energy-estimate}
are valid if the condition \eqref{eq:cond-eps-tau} holds.
Therefore, the parameter $\tau$ can be chosen of the order $\mathcal{O}\left(\e^{-k}\right)$ for some $k>0$ and the results of this section hold true by choosing $\theta>\max\{2,k+1\}$;
in particular, the estimate \eqref{eq:E-Psi} is valid with $\mu=\theta+k$.
On the other hand, if either $\tau$ is independent on $\e$ or $\tau\to0^+$ as $\e\to0^+$, we can choose any $\theta>2$ and the estimate \eqref{eq:E-Psi} is valid with $\mu=\theta$.

In general, if $\tau=\tau(\e)$ for some function $\tau:\mathbb{R}^+\rightarrow\mathbb{R}^+$,
then we can prove the results of Propositions \ref{prop:estimate-channel} and \ref{prop:energy-estimate} by working with the energy
\begin{equation*}
	E^{\bm\xi}[\wt,\vt]:=\frac12A^{\bm\xi}(\wt)+\frac\tau2\|\vt\|^2+f(\e)\tau\langle\wt,\vt\rangle,
\end{equation*}
where $f:\mathbb{R}^+\to\mathbb{R}^+$ is a function such that $f(\e)\tau(\e)/\e\to0^+$ and $f(\e)/\e^2\to0^+$ as $\e\to0^+$.
\end{rem}
Now, we have all the tools to prove our main result.
\begin{proof}[Proof of Theorem \ref{thm:main}]
Let $(\ut_0,\ut_1)\in\,\stackrel{\circ}{\mathcal{Z}}_{{}_{\Gamma,\rho}}$ and let $(\ut,\vt)\in\mathcal{Z}_{{}_{\Gamma,\rho}}$ 
for $t\in[0,T_\varepsilon]$ be the solution to \eqref{eq:integrated-CH-sys}. 
Then, $\ut=\ut^{\bm\xi}+\wt$ and $(\wt,\vt,\bm\xi)\in\hat{\mathcal{Z}}_{{}_{\Gamma,\rho}}$ solves the system \eqref{eq:w-v-xi} for $t\in[0,T_\varepsilon]$. 
We have already seen that the property \eqref{eq:u-uh,h'} holds.
Assume that $T_\varepsilon$ is maximal and apply Proposition \ref{prop:energy-estimate}; from \eqref{eq:E-Psi}, it follows that
\begin{equation*}
	\frac d{dt}\Bigl\{\exp(\eta\varepsilon^\mu t)(E^{\bm{\xi}}[\wt,\vt]-\Gamma\e^{-2}\Psi(\bm h))\Bigr\}\leq0,
	\quad \qquad t\in[0,T_\varepsilon]
\end{equation*}
and so,
\begin{equation*}
	\exp(\eta\varepsilon^\mu t)\left\{E^{\bm{\xi}}[\wt,\vt]-\Gamma\e^{-2}\Psi(\bm h)\right\}(t)\leq\left\{E^{\bm{\xi}}[\wt,\vt]-\Gamma\e^{-2}\Psi(\bm h)\right\}(0)<0,
	\qquad  t\in[0,T_\varepsilon].
\end{equation*}
Therefore, $(\ut,\vt)$ remains in the channel $\mathcal{Z}_{{}_{\Gamma,\rho}}$ while $\bm h=\mathbf{G}(\bm\xi)\in\partial\Omega_\rho\in\Omega_\rho$ 
 and if $T_\varepsilon<+\infty$ is maximal, then $\bm h(T_\varepsilon)\in\partial\Omega_\rho$, that is
\begin{equation}\label{hfrontiera}
	h_j(T_\varepsilon)-h_{j-1}(T_\varepsilon)=\varepsilon/\rho \qquad \textrm{for some } j.
\end{equation}
From \eqref{eq:u-uh,h'}  it follows that for all $t\in[0,T_\varepsilon]$, one has
\begin{equation}\label{dhmax}
	|h_j(t)-h_j(0)|\leq C\e^{-2}\tau^{-1/2}\exp(-Al^{\bm h(t)}/\varepsilon)t \qquad \textrm{for any } j=1,\dots,N+1,
\end{equation} 
where $l^{\bm h(t)}$ is the minimum distance between layers at the time $t$.
Combining \eqref{hfrontiera} and \eqref{dhmax},  we obtain 
\begin{equation*}
	\varepsilon/\rho\geq l^{\bm h(0)}-2C\e^{-2}\tau^{-1/2}\exp(-A/\rho)T_\varepsilon.
\end{equation*}
Hence, using \eqref{eq:triangle} we have
\begin{equation*}
	T_\varepsilon\geq C\bigl(\ell^{\bm h(0)}-\varepsilon/\rho\bigr)\e^2\tau^{1/2}\exp(A/\rho)\geq 
	C\bigl(\ell^{\bm h(0)}-\varepsilon/\rho\bigr)\e^2\tau^{1/2}\exp(A\delta/\varepsilon),
\end{equation*}
and the proof is complete.
\end{proof}

\section{Layer dynamics}\label{sec:layers}
As we have seen in the previous section, there exist metastable states for the hyperbolic Cahn--Hilliard equation \eqref{eq:hyp-CH},
that are approximately equal to $+1$ or $-1$ except near $N+1$ transition points moving with exponentially small velocity.
The aim of this section is to derive and study a system of ODEs describing the movement of the transition layers.
Precisely, after deriving a system of ODEs from \eqref{eq:xi-compact}, 
we will compare such system with the one obtained in the case of the classic Cahn--Hilliard equation \eqref{eq:CH} by studying, in particular, the limit as $\tau\to0$.
\subsection{Equations of transition layers}
In order to derive the system of ODEs, we use the approximation $(\wt,\vt)\approx(0,\sum_{j=1}^N\xi'_j\ut_j^{\xi})$; substituting $\wt=0$ in \eqref{eq:xi} we get
\begin{equation*}
	\sum_{i=1}^N\langle \ut_i^{\bm\xi},E_j^{\bm\xi}\rangle\xi'_i=\langle v,E^{\bm\xi}_j\rangle, \qquad j=1,\dots,N.
\end{equation*}
In order to eliminate the variable $v$, let us differentiate and multiply by $\tau$ the latter equation:
\begin{align*}
	\tau\sum_{i,l=1}^N & \bigl(\langle\ut^{\bm\xi}_{il},E^{\bm\xi}_j\rangle+\langle\ut^{\bm\xi}_i,E^{\bm\xi}_{jl}\rangle\bigr)\xi'_l\xi'_i
	+\tau\sum_{i=1}^N\langle\ut^{\bm\xi}_i,E^{\bm\xi}_j\rangle\xi''_i\\
	& =-\langle\mathcal{L}(\ut^{\bm\xi}),E^{\bm\xi}_j\rangle-\langle v,E^{\bm\xi}_j\rangle+
	\tau\sum_{l=1}^N\langle v,E^{\bm\xi}_{jl}\rangle\xi'_l, \qquad j=1,\dots,N.
\end{align*}
Using the approximation $\vt\approx\sum_{j=1}^N\xi'_j\ut_j^{\xi}$, we obtain
\begin{equation}\label{eq:xi-sum1}
	\tau\sum_{i=1}^N\langle\ut_i^{\bm\xi},E^{\bm\xi}_j\rangle\xi''_i+\sum_{i=1}^N\langle\ut_i^{\bm\xi},E^{\bm\xi}_j\rangle\xi'_i+\tau\sum_{i,l=1}^N\langle\ut_{il}^{\bm\xi},E^{\bm\xi}_j\rangle\xi'_i\xi'_l=
	\langle\mathcal{L}\big(\ut^{\bm\xi}\big),E^{\bm\xi}_j\rangle,
\end{equation}
for $j=1,\dots,N$.
In order to simplify \eqref{eq:xi-sum1}, let us compute the terms  $a_{ij}=\langle\ut_i^{\bm\xi},E^{\bm\xi}_j\rangle$, $\langle\ut_{il}^{\bm \xi},E^{\bm\xi}_j\rangle$
and $\langle\mathcal{L}\big(\ut^{\bm\xi}\big),E^{\bm\xi}_j\rangle$. 
The formula for $a_{ij}$ is given in \eqref{eq:aij} and implies that the matrix $(a_{ij})\in\mathbb{R}^{N\times N}$ has the form
\begin{equation*}
	(a_{ij})=\left(\begin{array}{ccccc} 4l_2 & 0 & 0 & \dots & 0\\
	-4l_3 & 4l_3 & 0 & \dots & 0\\
	4l_4 & -4l_4 & 4l_4 & \dots & 0\\
	\dots & \dots & \dots & \dots & \dots\\
	(-1)^{N-1}4l_{N+1} & (-1)^{N-2}4l_{N+1} & (-1)^{N-3}4l_{N+1} & \dots & 4l_{N+1}
	\end{array}\right)+\mathcal{O}(\e),
\end{equation*}
with inverse
\begin{equation*}
	(a_{ij})^{-1}:=\left(\begin{array}{cccccc} \displaystyle{\frac1{4l_2}} & 0 & 0 & \dots & 0 & 0\\
	\displaystyle{\frac1{4l_2}} & \displaystyle{\frac1{4l_3}} & 0 & \dots & 0 & 0\\
	0 & \displaystyle\frac1{4l_3} & \displaystyle\frac1{4l_4} & \dots & 0 & 0\\
	\dots & \dots & \dots & \dots & \dots\\
	0 & 0 & 0 & \dots &  \displaystyle\frac1{4l_N} &  \displaystyle\frac1{4l_{N+1}}\\
	\end{array}\right)+\mathcal{O}(\e).
\end{equation*}
Next, for Lemma \ref{lem:u^h_j}, \eqref{eq:AC-CH} and the definition $E^{\bm\xi}_j$ \eqref{eq:Ej-Qj}, we have
\begin{equation}\label{eq:L,Ei}
	\begin{aligned}
		\langle\mathcal{L}\big(\ut^{\bm\xi}\big),E^{\bm\xi}_j\rangle & =\langle\mathcal{L}^{AC}\big(u^{\bm\xi}\big),u^{\bm h}_j+u^{\bm h}_{j+1}-Q'_{j}\rangle
		= \langle\mathcal{L}^{AC}\big(u^{\bm\xi}\big),u^{\bm h}_j+u^{\bm h}_{j+1}\rangle+\mathcal{O}(e^{-c/\e})\\
		& =\int_{{I_j}\cup{I_{j+1}}}\left(\e^2 u^{\bm h}_{xx}(x)-F'(u^{\bm h}(x))\right)u^{\bm h}_x(x)\,dx+\mathcal{O}(e^{-c/\e})\\
		& =\alpha^{j+2}-\alpha^{j}+\mathcal{O}(e^{-c/\e}),
	\end{aligned}
\end{equation}
for $j=1,\dots,N$.
Finally, let us compute the terms $\langle\ut_{il}^{\bm \xi},E^{\bm\xi}_j\rangle$, by using the formulas \eqref{eq:ut-j-h} and \eqref{eq:Ej} for $u_i^{\bm h}$ and $E^{\bm\xi}_j$, respectively.
In what follows, we omit the tedious, but straightforward computation of the derivatives of the exponentially small terms,
because one can prove (using the bounds in \cite{Bates-Xun1,Bates-Xun2,Carr-Pego,Carr-Pego2}) that they remain exponentially small in $\e$. 
Therefore, differentiating the identities 
\begin{equation*}
	\ut^{\bm\xi}_i=\ut_i^{\bm h}+\ut_{N+1}^{\bm h}\frac{\partial h_{N+1}}{\partial h_i}, \qquad  \mbox{ and } \qquad \frac{\partial h_{N+1}}{\partial h_i}=(-1)^{N-i}+\mathcal{O}(e^{-c/\e}),
\end{equation*}
 we obtain
\begin{equation*}
	\ut^{\bm\xi}_{il}=\ut_{il}^{\bm h}+(-1)^{N-l}\ut_{i,N+1}^{\bm h}+(-1)^{N-i}\ut_{N+1,l}^{\bm h}
	+(-1)^{i+l}\ut_{N+1,N+1}^{\bm h}+\mathcal{O}(e^{-c/\e}).
\end{equation*}
From \eqref{eq:ut-j-h} and the formula for $u_i^{\bm h}$ of Lemma \ref{lem:u^h_j}, it follows that
\begin{equation*}
	\ut_{ii}^{\bm h}(x)=\begin{cases}
	u^{\bm h}_x(x)+\mathcal{O}(e^{-c/\e}), \qquad \qquad &x\in I_i,\\
	 \mathcal{O}(e^{-c/\e}), & \mbox{otherwise},
	\end{cases}
\end{equation*}
and $\ut_{il}^{\bm h}(x)=e$, $i\neq l$, for $i=1,\dots,N+1$. 
Hence, we have 
\begin{align}
	\ut^{\bm\xi}_{ii}(x)& =\begin{cases}
	u^{\bm h}_x(x)+\mathcal{O}(e^{-c/\e}), \qquad \qquad & x\in I_i\cup I_{N+1},\\
	 \mathcal{O}(e^{-c/\e}), & \mbox{otherwise},
	\end{cases} \qquad & \mbox{ for } i=1,\dots,N,\label{eq:ut-jj-xi}
	\\
	\ut^{\bm\xi}_{il}(x)& =\begin{cases}
	(-1)^{i+l}u^{\bm h}_x(x)+\mathcal{O}(e^{-c/\e}), \qquad \qquad & x\in I_{N+1},\\
	 \mathcal{O}(e^{-c/\e}), & \mbox{otherwise},
	\end{cases} & \mbox{ for } i\neq l. \label{eq:ut-jl-xi}
\end{align}
Thanks to the formulas \eqref{eq:Ej}, \eqref{eq:ut-jj-xi} and \eqref{eq:ut-jl-xi}, we can compute the quantities $\langle\ut_{il}^{\bm\xi},E^{\bm\xi}_j\rangle$.
Let us start with the case $i=l=j\neq N$; for \eqref{eq:Ej} and \eqref{eq:ut-jj-xi}, we deduce
\begin{align*}
	\langle\ut_{ii}^{\bm\xi},E^{\bm\xi}_i\rangle&=\int_0^1\ut_{ii}^{\bm\xi}(x)E^{\bm\xi}_i(x)\,dx=
	\int_{h_{i-1/2}}^{h_{i+1/2}}u^{\bm h}_x(x)\left[(-1)^i-u^{\bm h}(x)\right]dx+\mathcal{O}(e^{-c/\e})\\
	&=-\frac12\left[(-1)^i-u^{\bm h}(x)\right]^2\bigg|_{h_{i-1/2}}^{h_{i+1/2}}+\mathcal{O}(e^{-c/\e})=-2+\mathcal{O}(e^{-c/\e}),
\end{align*}
for $i=1,\dots,N-1$. 
In the case $i=l=j=N$, we have
\begin{align*}
	\langle\ut_{NN}^{\bm\xi},E^{\bm\xi}_N\rangle&=\int_0^1\ut_{NN}^{\bm\xi}(x)E^{\bm\xi}_N(x)\,dx=
	\int_{h_{N-1/2}}^{1}u^{\bm h}_x(x)\left[(-1)^N-u^{\bm h}(x)\right]dx+\mathcal{O}(e^{-c/\e})\\
	&=-\frac12\left[(-1)^N-u^{\bm h}(x)\right]^2\bigg|_{h_{N-1/2}}^{1}+\mathcal{O}(e^{-c/\e})= \mathcal{O}(e^{-c/\e}).
\end{align*}
The latter equality together with the expression for $(a_{ij})^{-1}$ and \eqref{eq:L,Ei} gives the equation for $\xi$ in the case $N=1$ (two layers): 
equation \eqref{eq:xi-sum1} in the case $N=1$ becomes
\begin{equation*}
	\tau\xi''+\xi'=\frac{1}{4l_2}(\alpha^{3}-\alpha^{1}).
\end{equation*}
Consider now the case $i=l=j+1$, $j\neq N$ with $N>1$; for the formulas \eqref{eq:Ej} and \eqref{eq:ut-jj-xi}, we infer
\begin{align*}
	\langle\ut_{j+1,j+1}^{\bm\xi},E^{\bm\xi}_j\rangle&=\int_0^1\ut_{j+1,j+1}^{\bm\xi}(x)E^{\bm\xi}_j(x)\,dx=
	\int_{h_{j+1/2}}^{h_{j+3/2}}u^{\bm h}_x(x)\left[(-1)^j-u^{\bm h}(x)\right]dx+\mathcal{O}(e^{-c/\e})\\
	&=-\frac12\left[(-1)^j-u^{\bm h}(y)\right]^2\bigg|_{h_{j+1/2}}^{h_{j+3/2}}+\mathcal{O}(e^{-c/\e})=2+\mathcal{O}(e^{-c/\e}),
\end{align*}
If $j\neq N$ and either $i=l\neq j, j+1$ or $i\neq l$, then all the terms $\langle\ut_{il}^{\bm\xi},E^{\bm\xi}_j\rangle$ are negligible for \eqref{eq:Ej} and \eqref{eq:ut-jl-xi}.
In conclusion, for $j\neq N$, we have
\begin{equation*}
	\langle\ut_{il}^{\bm\xi},E^{\bm\xi}_j\rangle=\mathcal{O}(e^{-c/\e})+\begin{cases}
	-2, \qquad \qquad & i=l=j,\\
	2, & i=l=j+1,\\
	0, &\mbox{otherwise}.
	\end{cases}
\end{equation*}
Hence, the first $N-1$ equations of \eqref{eq:xi-sum1} become
\begin{equation*}
	\sum_{i=1}^{N}(\tau\xi''_i+\xi'_i)a_{ij}
	+2\tau\left[\left(\xi'_{j+1}\right)^2-\left(\xi'_j\right)^2\right]=\alpha^{j+2}-\alpha^{j},   
\end{equation*}
for $j=1,\dots,N-1$. 
The last equation of \eqref{eq:xi-sum1} is more difficult because the functions $\ut_{il}^{\bm\xi}$ and $E^{\bm\xi}_N$ are not negligible in $I_{N+1}$.
We have already seen that $\langle\ut_{NN}^{\bm\xi},E^{\bm\xi}_N\rangle=e$, for the other terms we have
\begin{align*}
	\langle\ut_{ii}^{\bm\xi},E^{\bm\xi}_N\rangle&=\int_0^1\ut_{ii}^{\bm\xi}(x)E^{\bm\xi}_N(x)\,dx=
	\int_{h_{N+1/2}}^{1}u^{\bm h}_x(x)\left[(-1)^N-u^{\bm h}(x)\right]dx+\mathcal{O}(e^{-c/\e})\\
	&=-\frac12\left[(-1)^N-u^{\bm h}(x)\right]^2\bigg|_{h_{N+1/2}}^{1}+\mathcal{O}(e^{-c/\e})=2+\mathcal{O}(e^{-c/\e}),
\end{align*}
for $i=1,\dots, N-1$, and
\begin{align*}
	\langle\ut_{il}^{\bm\xi},E^{\bm\xi}_N\rangle&=\int_0^1\ut_{il}^{\bm\xi}(x)E^{\bm\xi}_N(x)\,dx=
	(-1)^{i+l}\int_{h_{N+1/2}}^{1}u^{\bm h}_x(x)\left[(-1)^N-u^{\bm h}(x)\right]dx+\mathcal{O}(e^{-c/\e})\\
	&=-\frac{(-1)^{i+l}}2\left[(-1)^N-u^{\bm h}(x)\right]^2\bigg|_{h_{N+1/2}}^{1}+\mathcal{O}(e^{-c/\e})=2(-1)^{i+l}+\mathcal{O}(e^{-c/\e}),
\end{align*}
for $i\neq l$.
Therefore, 
\begin{equation*}
	\langle\ut_{il}^{\bm\xi},E^{\bm\xi}_N\rangle=\mathcal{O}(e^{-c/\e})+\begin{cases}
	0, \qquad \qquad \qquad & i=l=N,\\
	2, & i=l\neq N,\\
	2(-1)^{i+l}, &\mbox{otherwise}.
	\end{cases}
\end{equation*}
It follows that the last equation of \eqref{eq:xi-sum1} becomes
\begin{equation*}
	\sum_{i=1}^N(\tau\xi''_i+\xi'_i)a_{ij}
	+2\tau\left[\sum_{i=1}^{N-1}\left(\xi'_i\right)^2+\sum_{i\neq l}(-1)^{i+l}\xi'_i\xi'_l\right]=\alpha^{N+2}-\alpha^N.
\end{equation*}
Since
\begin{align*}
	\sum_{i=1}^{N-1}\left(\xi'_i\right)^2+\sum_{i\neq l}(-1)^{i+l}\xi'_i\xi'_l&=\left(\sum_{i=1}^{N-1}(-1)^{N-j}\xi'_i\right)^2+2\xi'_N\sum_{i=1}^{N-1}(-1)^{N-j}\xi'_i\\
	&=\left(\sum_{i=1}^{N}(-1)^{N-i}\xi'_i-\xi'_N\right)\left(\sum_{i=1}^{N}(-1)^{N-i}\xi'_i+\xi'_N\right),
\end{align*}
we can rewrite
\begin{equation*}
	\sum_{i=1}^N(\tau\xi''_i+\xi'_i)a_{ij}
	+2\tau\left[\left(\sum_{i=1}^{N}(-1)^{N-i}\xi'_i\right)^2-\left(\xi'_N\right)^2\right]=\alpha^{N+2}-\alpha^N.
\end{equation*}
By applying the inverse matrix $(a_{ij})^{-1}$, we obtain the following equation for $\bm\xi$:
\begin{equation*}
	\begin{aligned}
		\tau\xi_1''+\xi_1'+\frac{\tau}{2l_2}Q(\xi'_2,\xi'_1)&=P_1(\bm h),\\
		\tau\xi_i''+\xi_i'+\frac{\tau}{2l_i}Q(\xi'_i,\xi'_{i-1})+\frac{\tau}{2l_{i+1}}Q(\xi'_{i+1},\xi'_i)&=P_{i-1}(\bm h)+P_i(\bm h), \\
		&\quad i=2,\dots,N-1, \\
		\tau\xi''_N+\xi'_N+\frac\tau{2l_N}Q(\xi'_N,\xi'_{N-1})+\frac\tau{2l_{N+1}}Q\left(\sum_{j=1}^{N}(-1)^{N-j}\xi'_j,\xi'_N\right)&=
		P_{N-1}(\bm h)+P_N(\bm h),
	\end{aligned}
\end{equation*}
where we introduced the functions 
\begin{equation}\label{eq:Q}
	Q(x,y):=x^2-y^2, \qquad\mbox{ and}\qquad  P_i(\bm h):=\frac{1}{4l_{i+1}}(\alpha^{i+2}-\alpha^{i}), \qquad i=1,\dots,N.
	%Q_N(\bm{h}')&:=\frac{1}{2l_{N+1}}\left[\left(\sum_{j=1}^{N-1}h'_j\right)^2+2h'_N\sum_{j=1}^{N-1}h'_j\right],
\end{equation}
Therefore, we derived the equation for $\bm\xi=(\xi_1,\dots,\xi_N)$; 
recall that the transition points are located at $\bm h=(h_1,\dots,h_N,h_{N+1})$ and that $\xi_i=h_i$ for $i=1,\dots,N$; 
the position of the last point $h_{N+1}$ is determined by the other points $(h_1,\dots,h_N)$ for the conservation of the mass.
In order to write the equation for $\bm h=(h_1,\dots,h_N,h_{N+1})$, which is more natural, symmetric and easy to handle,
we use \eqref{eq:h'_N+1} by neglecting the exponentially smallest terms; 
thus, we consider the approximations
\begin{equation*}
	h'_{N+1}\approx\sum_{j=1}^N(-1)^{N-j}h'_j, \qquad \qquad h''_{N+1}\approx\sum_{j=1}^N(-1)^{N-j}h''_j.
\end{equation*}
We can now write the equation for $\bm h=(h_1,\dots,h_N,h_{N+1})$.
In the case $N=1$ (two layers) we have
\begin{equation}\label{eq:h2layers}
	\begin{aligned}
		\tau h_1''+h_1'=\frac{1}{4l_2}(\alpha^3-\alpha^1),\\
		\tau h_2''+h_2'=\frac{1}{4l_2}(\alpha^3-\alpha^1).
	\end{aligned}
\end{equation}
%For $N=2$ (three layers) we have
%\begin{equation}\label{eq:h3layers}
%	\begin{aligned}
%		\tau h_1''+h_1'+\frac{\tau}{2l_2}\left[\left(h_2'\right)^2-\left(h_1'\right)^2\right]&=\frac{1}{4l_2}(\alpha^3-\alpha^1),\\
%		\tau h_2''+h_2'+\frac{\tau}{2l_2}\left[\left(h_{2}'\right)^2-\left(h_{1}'\right)^2\right]+\frac{\tau}{2l_{3}}\left[\left(h'_3\right)^2-\left(h_2'\right)^2\right]&=
%		\frac{1}{4l_2}(\alpha^{3}-\alpha^{1})+\frac{1}{4l_{3}}(\alpha^{4}-\alpha^{2}),\\
%		\tau h_3''+h_3'+\frac{\tau}{2l_{3}}\left[\left(h'_3\right)^2-\left(h_2'\right)^2\right]&=\frac{1}{4l_{3}}(\alpha^{4}-\alpha^{2}).
%	\end{aligned}
%\end{equation}
%The equation for $h_3$ has been obtained using that $h_3'=h_2'-h_1'$.
In general, for $N\geq2$ the equations are
\begin{equation}\label{eq:hN+1layers}
	\begin{aligned}
		\tau h_1''+h_1'+\frac{\tau}{2l_2}Q(h'_2,h'_1)&=P_{1}(\bm h),\\
		\tau h_i''+h_i'+\frac{\tau}{2l_{i}}Q(h'_i,h'_{i-1})+\frac\tau{2l_{i+1}}Q(h'_{i+1},h'_i)&=P_{i-1}(\bm h)+P_i(\bm h), \qquad \, i=2,\dots,N, \\
%		\tau h_N''+h_N'+\tau\,Q_{N-1}(\bm{h}')+\tau\,Q_N(\bm{h}')&=P_{N-1}(\bm{h})+P_{N}(\bm{h}),\\
		\tau h_{N+1}''+h_{N+1}'+\frac{\tau}{2l_{N+1}}Q(h'_{N+1},h'_N)&=P_{N}(\bm h).
	\end{aligned}		
\end{equation}
Equations \eqref{eq:h2layers} and \eqref{eq:hN+1layers} imply the following equations for the interval length $l_j$ (remember that $l_1=h_1-h_0=2h_1$ and $l_{N+2}=h_{N+2}-h_{N+1}=2(1-h_{N+1})$).
For $N=1$ one has
\begin{equation}\label{eq:l2layers}
	\begin{aligned}
		\tau l_1''+l_1'&=\frac{1}{2l_2}(\alpha^3-\alpha^1),\\
		\tau l_2''+l'_2&=0,\\
		\tau l_3''+l_3'&=-\frac{1}{2l_2}(\alpha^3-\alpha^1).
	\end{aligned}
\end{equation}
For $N=2$:
\begin{equation}\label{eq:l3layers}
	\begin{aligned}
		\tau l_1''+l_1'+\frac{\tau}{l_2}\left[\left(h_2'\right)^2-\left(h_1'\right)^2\right]&=\frac{1}{2l_2}(\alpha^3-\alpha^1),\\
		\tau l_2''+l'_2+\frac{\tau}{2l_{3}}\left[\left(h'_3\right)^2-\left(h_2'\right)^2\right]&=\frac{1}{4l_{3}}(\alpha^{4}-\alpha^{2}),\\
		\tau l_3''+l_3'-\frac{\tau}{2l_2}\left[\left(h_{2}'\right)^2-\left(h_{1}'\right)^2\right]&=-\frac{1}{4l_2}(\alpha^{3}-\alpha^{1}),\\
		\tau l_4''+l_4'-\frac{\tau}{l_{3}}\left[\left(h'_3\right)^2-\left(h_2'\right)^2\right]&=-\frac{1}{2l_{3}}(\alpha^{4}-\alpha^{2}).
	\end{aligned}
\end{equation}
In general, for $N\geq3$:
\begin{equation}\label{eq:lN+1layers}
	\begin{aligned}
		\tau l_1''+l_1'+\frac\tau{l_2} Q(h'_2,h'_1)&=2P_1(\bm h),\\
		\tau l_2''+l_2'+\frac\tau{2l_3}Q(h'_3,h_2)&=P_2(\bm h),\\
		\tau l_i''+l_i'+\frac\tau{2l_{i+1}}Q(h'_{i+1},h'_i)-\frac{\tau}{2l_{i-1}}Q(h'_{i-1},h'_{i-2})&=P_{i}(\bm h)-P_{i-2}(\bm h), \qquad  i=3,\dots,N, \\
%		\tau l_N''+l_N'+\tau\,Q_{N}(\bm{h}')-\tau\,Q_{N-2}(\bm{h}')&=P_{N}(\bm{h})-P_{N-2}(\bm{h}),\\
		\tau l_{N+1}''+l_{N+1}'-\frac{\tau}{2l_{N}}Q(h'_N,h'_{N-1})&=-P_{N-1}(\bm h),\\
		\tau l_{N+2}''+l_{N+2}'-\frac\tau{l_{N+1}}Q(h'_{N+1},h'_N)&=-2P_{N}(\bm h).
	\end{aligned}
\end{equation}
Observe that $l_1/2$ and $l_{N+2}/2$ are the distances of $h_1$ and $h_{N+1}$ from the boundary points $0$ and $1$, respectively.
Let $L_-$ and $L_+$ be the length of all the intervals where the solution is approximately $-1$ and $+1$, respectively; namely
\begin{align*}
	L_-:&=\frac{l_1}2+\sum_{i=1}^{N/2} l_{2i+1}, \qquad \qquad \qquad  &L_+&=\sum_{i=1}^{N/2} l_{2i}+\frac{l_{N+2}}{2}, \qquad  &\mbox{ if }\, N \mbox{ is even},\\
	L_-:&=\frac{l_1}2+\sum_{i=1}^{(N-1)/2} l_{2i+1}+\frac{l_{N+2}}{2},   &L_+&=\sum_{i=1}^{(N+1)/2} l_{2i},   &\mbox{ if }\, N \mbox{ is odd}.
\end{align*}
It follows that these quantities satisfy 
\begin{equation}\label{eq:L_pm}
	\tau L_{\pm}''+L_{\pm}'=0.
\end{equation}

\subsection{Comparison with the classic Cahn--Hilliard equation}
In this subsection, we study the equations describing the movement of the transition points derived above,
and we analyze the differences with the corresponding equations valid for the classic Cahn--Hilliard equation \eqref{eq:CH}.
Rewrite the equations \eqref{eq:h2layers} and \eqref{eq:hN+1layers} in a compact form:
in the case of two layers ($N=1$), see equations \eqref{eq:h2layers}, we get
\begin{equation}\label{eq:h2layers-compact}
	\tau\bm h''+\bm h'=\bm{\mathcal{P}}(\bm h),
\end{equation}
where $\bm h=(h_1,h_2)$ and $\bm{\mathcal{P}}:\R^2\rightarrow\R^2$ is defined by 
$$
	\mathcal P_i(h_1,h_2):=\frac{\alpha^3-\alpha^1}{4(h_2-h_1)}, \qquad \qquad i=1,2.
$$
In the case of $N+1$ layers with $N\geq2$, we rewrite \eqref{eq:hN+1layers} as
\begin{equation}\label{eq:hN+1layers-compact}
	\tau \bm h''+\bm h'+\tau\bm{\mathcal{Q}}(\bm h,\bm h')=\bm{\mathcal P}(\bm h),
\end{equation}
where $\bm h=(h_1,\dots,h_{N+1})$ and $\bm{\mathcal P}:\R^{N+1}\rightarrow\R^{N+1}$ is defined by
\begin{equation}\label{eq:P(h)}
	\bm{\mathcal P}(\bm h):=\left(\begin{array}{c}
	\displaystyle\frac{\alpha^3-\alpha^1}{4(h_2-h_1)}\vspace{0.1cm}\\
	\displaystyle{\frac{\alpha^3-\alpha^1}{4(h_2-h_1)}+\frac{\alpha^4-\alpha^2}{4(h_3-h_2)}}\\
	\vdots \\ \vdots \\
	\displaystyle{\frac{\alpha^{N+1}-\alpha^{N-1}}{4(h_N-h_{N-1})}+\frac{\alpha^{N+2}-\alpha^N}{4(h_{N+1}-h_N)}}\vspace{0.1cm}\\
	\displaystyle{\frac{\alpha^{N+2}-\alpha^N}{4(h_{N+1}-h_N)}}
	\end{array}\right),
\end{equation}
and $\bm{\mathcal{Q}}:\R^{N+1}\times\R^{N+1}\rightarrow\R^{N+1}$ is
\begin{equation}\label{eq:Q(h)}
	\bm{\mathcal Q}(\bm h,\bm h'):=\left(\begin{array}{c}
	\displaystyle\frac{\left(h_2'\right)^2-\left(h_1'\right)^2}{2(h_2-h_1)},\vspace{0.1cm}\\
 	\displaystyle\frac{\left(h_2'\right)^2-\left(h_1'\right)^2}{2(h_2-h_1)}+\frac{\left(h_3'\right)^2-\left(h_2'\right)^2}{2(h_3-h_2)}\\
	\vdots\\ \vdots \\
	\displaystyle\frac{\left(h_N'\right)^2-\left(h_{N-1}'\right)^2}{2(h_N-h_{N-1})} 	   
	+\frac{\left(h_{N+1}'\right)^2-\left(h_N'\right)^2}{2(h_{N+1}-h_N)}\vspace{0.1cm}\\
	\displaystyle\frac{\left(h_{N+1}'\right)^2-\left(h_N'\right)^2}{2(h_{N+1}-h_N)}
	\end{array}\right).
\end{equation}
Both in the case \eqref{eq:h2layers-compact} and in the case \eqref{eq:hN+1layers-compact}, taking formally the limit as $\tau\to0^+$ 
one obtains the system describing the motion of the transition layers in the classic Cahn--Hilliard equation \eqref{eq:CH} (see equations (4.36) in \cite{Bates-Xun2}).
%Let us briefly recall the results on the motion of the transition layer in the classic Cahn--Hilliard equation \eqref{eq:CH}.
Indeed, in \cite{Bates-Xun2} the authors derived the following system of ODEs to approximately describe the motion of the transition points $h_1,h_2,\dots,h_{N+1}$ when they are well separated:
\begin{equation}\label{eq:h_i-CahnHilliard}
	\begin{aligned}
		h_1'&=\frac{1}{4l_2}(\alpha^3-\alpha^1),\\
		h_j'&=\frac{1}{4l_j}(\alpha^{j+1}-\alpha^{j-1})+\frac{1}{4l_{j+1}}(\alpha^{j+2}-\alpha^j), \qquad \qquad \qquad j=2,\dots,N\\
		h_{N+1}'&=\frac{1}{4l_{N+1}}(\alpha^{N+2}-\alpha^N).
	\end{aligned}		
\end{equation}
Let us briefly describe the behavior of the solutions to \eqref{eq:h_i-CahnHilliard} when $F$ is an even function 
and $h_i$, $h_{i-1}$ are the closest transition points at time $t=0$, namely assume that there exists a unique $i\in\{1,\dots,N+2\}$ such that
\begin{equation*}
	l_i(0)<l_j(0), \quad j\neq i, \quad j=1,\dots,N+2.
\end{equation*}
In this case, we can use Remark \ref{rem:alfa} and from the estimate \eqref{eq:alfa^j-alfa^i} it follows that $\alpha^i\gg\alpha^j$  for all $j\neq i$, 
and the terms $\alpha^j$ with $j\neq i$ are exponentially small with respect to $\alpha^i$ as $\e\to0^+$.
As a consequence, we can describe the motion of the transition layers in the case of the Cahn--Hilliard equation \eqref{eq:CH} as follows.
In the case $N=1$, the two transition points $h_1$ and $h_2$ move to the right (respectively left) if $l_3=2(1-h_2)<l_1=2h_1$ (respectively $l_3>l_1$) and we have $h'_1\approx h'_2$;
thus, the points move together in an almost rigid way, they move in the same direction at approximately the same speed. 
In the case $N=2$, we have two transitions points moving in the same direction at approximately the same speed $v$, 
while the speed of the third one is exponentially small with respect to $v$, 
and so, the third point is essentially static.
Finally, consider the case $N\geq3$ with $i\in\{3,\dots,N\}$;
the term $\alpha^i$ appears in the equations for $h_{i-2}$, $h_{i-1}$, $h_{i}$ and $h_{i+1}$,
and so we have four points moving at approximately the same speed, 
while all the other layers remain essentially stationary in time.
Precisely, we have  
\begin{equation*}
	h'_{i-2}>0, \; h'_{i-1}>0, \; h'_{i}<0, \; h'_{i+1}<0, \; h'_j=\mathcal{O}\left(e^{-C/\e}h'_i\right) \, \mbox{ for } j\notin\{i-2,i-1,i,i+1\}.
\end{equation*}
Roughly speaking, the system \eqref{eq:h_i-CahnHilliard} shows that the shortest distance between layers decreases:
the closest layers move towards each other, each being followed by its nearest transition point from ``behind'', at approximately the same speed, 
until the points $h_i$ and $h_{i-1}$ are sufficiently close.
Hence, the loss of the mass due to the annihilation of the transitions at $h_{i-1}$ and $h_i$ is compensated by the movement of the nearest neighbor $h_{i-2}$ and $h_{i+1}$.
This property, due to the conservation of the mass, is a fundamental difference with respect to the Allen--Cahn equation \eqref{eq:AC}.
For such equation, Carr and Pego \cite{Carr-Pego} derived the following equations for the transition points $h_j$:
\begin{equation*}
	h'_j=C\e\left(\alpha^{j+1}-\alpha^j\right), \qquad \qquad j=1,\dots,N+1,
\end{equation*}
where $C$ is a constant depending only on $F$.
Then, in the case of the Allen--Cahn equation, the closest layers move towards each other at approximately the same speed satisfying $|h'_i|\approx\e|\alpha^{i+1}-\alpha^i|$, 
while all the other points remain essentially stationary in time.

As it was already mentioned, system \eqref{eq:h_i-CahnHilliard} was derived in \cite[Section 4]{Bates-Xun2} in order to approximately describe 
the movement of the transition layers for the Cahn--Hilliard equation \eqref{eq:CH} until the points are well separated, with distance $l_j>\e/\rho$.
A detailed analysis of the motion of the layers for \eqref{eq:CH} can be also found in \cite{SunWard}, 
where the authors studied in details layer collapse events and presented many numerical simulations
confirming that the layer dynamics is closely described by \eqref{eq:h_i-CahnHilliard}.
However, system \eqref{eq:h_i-CahnHilliard} provides an accurate description of the motion of the points corresponding to the annihilating interval and its two nearest neighbors, 
but it may be slightly inaccurate for other layers.
For example, in \cite{SunWard} it is showed that if $(h_{i-1},h_i)$ is the annihilating interval for some $i\in\{3,\dots,N\}$,
all the points $h_j$ with $j\notin\{i-2,i-1,i,i+1\}$ move at an algebraic slower speed in $\e$ than $h_i$.
In contrast, we saw that for \eqref{eq:h_i-CahnHilliard} the points $h_j$ move exponentially slower than the collapsing layers. 
Apart from that, system \eqref{eq:h_i-CahnHilliard} provides a good description of the layer dynamics for the classic Cahn--Hilliard equation \eqref{eq:CH}. 

In the case of the hyperbolic Cahn--Hilliard equation \eqref{eq:hyp-CH}, the movement of the layer is approximately describe by equations \eqref{eq:h2layers-compact} and \eqref{eq:hN+1layers-compact},
and so, we have to take in account the inertial term $\tau\bm h''$ and the quadratic term $\tau\bm{\mathcal{Q}}(\bm h,\bm h')$ (when $N\geq2$).
In the following, we shall compute some numerical solutions in order to analyze the differences between systems \eqref{eq:hN+1layers} and \eqref{eq:h_i-CahnHilliard}.
To do this, we use Proposition \ref{prop:alfa,beta} choosing $A_+=A_-=\sqrt2$ and $K_+=K_-=4$, which corresponds to the choice $F(u)=\frac14(u^2-1)^2$;
then, we use the approximation
\begin{equation*}
	\alpha^j\approx16\exp\left(-\frac{\sqrt2(h_j-h_{j-1})}{\e}\right).
\end{equation*}

The values of the initial data for the ODEs \eqref{eq:hN+1layers} depend on the choice of the initial datum $(u_0,u_1)$ for the PDE \eqref{eq:hyp-CH};
precisely, we assume that $u_0=u^{\bm{h^0}}$ for some $\bm{h^0}\in\Omega_\rho$, and so $\bm h(0)=\bm{h^0}$ represents the positions of the transition points at time $t=0$,
while the first $N$ components of $\bm h'(0)$ satisfy the third equation of \eqref{eq:w-v-xi}, and $h_{N+1}$ is given by \eqref{eq:h'_N+1} (for the conservation of the mass).
Therefore, we have
\begin{equation}\label{eq:initial-h}
	\begin{aligned}
		h'_1(0)&=\frac{1}{4l_2(0)}\langle\ut_1,E_1^{\bm\xi}\rangle+\mathcal{O}(\e\langle\ut_1,E_1^{\bm\xi}\rangle),\\
		h'_j(0)&=\frac{1}{4l_j(0)}\langle\ut_1,E_{j-1}^{\bm\xi}\rangle+\frac{1}{4l_{j+1}(0)}\langle\ut_1,E_j^{\bm\xi}\rangle+\mathcal{O}(\e\langle\ut_1,E_j^{\bm\xi}\rangle),
		\qquad \qquad j=2,\dots,N\\
		h'_{N+1}(0)&=\sum_{j=1}^N\frac{\partial h_{N+1}}{\partial h_j}h'_j(0)=\frac{1}{4l_{N+1}(0)}\langle\ut_1,E_N^{\bm\xi}\rangle+\mathcal{O}(\e\|\ut_1\|).
	\end{aligned}
\end{equation}
As we have previously done for the ODEs \eqref{eq:hN+1layers}, we consider equations \eqref{eq:initial-h} without the smallest terms $\mathcal{O}(\e\langle\ut_1,E_j^{\bm\xi}\rangle)$.
By reasoning as in the computation of \eqref{eq:L_pm},  we get $L_{\pm}'(0)=0$, and so $L_\pm(t)=0$ for all $t$ and this is consistent with the mass conservation.
In particular, let us stress that in the 2 layers case $(N=1)$ we have $h'_1(0)=h'_2(0)$. 
Finally, notice that choosing $\ut_1=\mathcal L(\ut_0)$ we deduce that $h'_j(0)$ satisfies \eqref{eq:h_i-CahnHilliard}.

We want to focus the attention on the role of the parameter $\tau$ and we consider the same initial data for \eqref{eq:hN+1layers} and \eqref{eq:h_i-CahnHilliard};
in particular for \eqref{eq:h2layers}-\eqref{eq:hN+1layers}, we choose $h'_j(0)$ satisfying \eqref{eq:h_i-CahnHilliard}, meaning that $h'_j(0)$ satisfy \eqref{eq:initial-h} with $\ut_1=\mathcal L(\ut_0)$.
Let us start with the case of $2$ layers.
Observe that $l_2=h_2-h_1$ satisfies \eqref{eq:l2layers} and since $h'_1(0)=h'_2(0)$, we have $l_2(t)=l_2(0)$ for any time $t$.
In the first example, we choose $\e=0.07$: in Table \ref{table:2layers_eps0.07} we show the numerical computation of the difference $h_1(t)-h_1(0)$ for different times $t$
and for different values of $\tau$ ($\tau=0$ corresponds to system \eqref{eq:h_i-CahnHilliard});
since $l_2$ is constant in time, we get $h_2(t)=s(t)+h_2(0)$; in Figure \ref{fig:h1} we show the graph of $h_1$ for $\tau=0$ and $\tau=50$.

\begin{table}[h!]
\vskip0.2cm
\begin{center}
\begin{tabular}{|c|c|c|c|}
\hline TIME $t$ &   $s(t)$, $\tau=0$  &  $s(t)$, $\tau=5$ & $s(t)$, $\tau=50$ \\ 
\hline $300$ & $-0.0128$ & $-0.0126$ & $-0.0113$  \\
\hline $600$ & $-0.0534$ & $-0.0497$ & $-0.0364$  \\
\hline $665$ & $-0.1240$ & $-0.0830$ & $-0.0475$   \\ \hline
\end{tabular}
\caption{The numerical computation of $s(t)=h_1(t)-h_1(0)$ for $\e=0.07$ and different values of $\tau$. 
The initial positions of the layers are $h_1(0)=0.31$, $h_2(0)=0.66$.}
\label{table:2layers_eps0.07}
\end{center}
\end{table}

\begin{figure}[htbp]
\centering
\includegraphics[height=5.5cm, width=6.9cm]{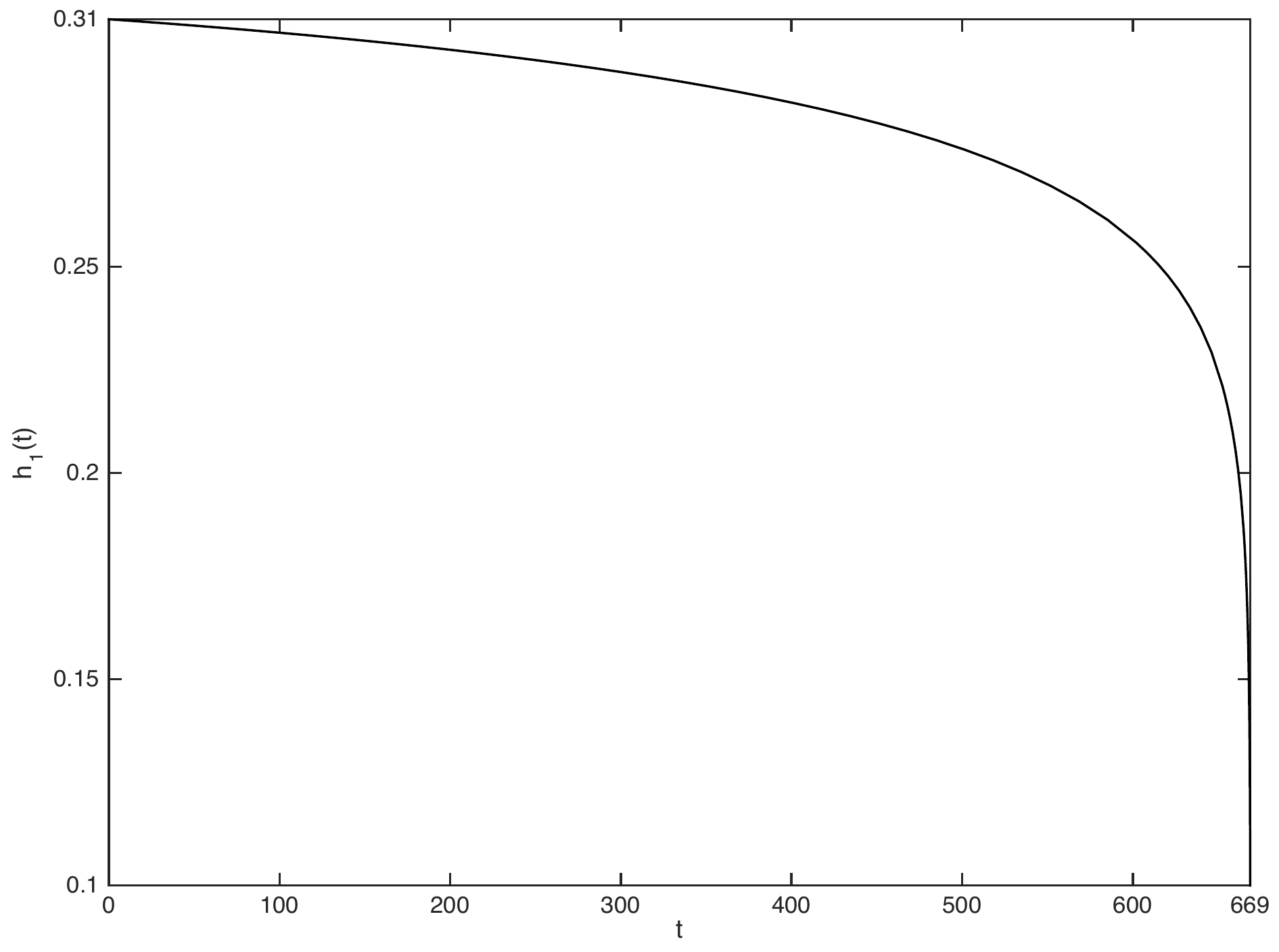}\quad
\includegraphics[height=5.5cm, width=6.9cm]{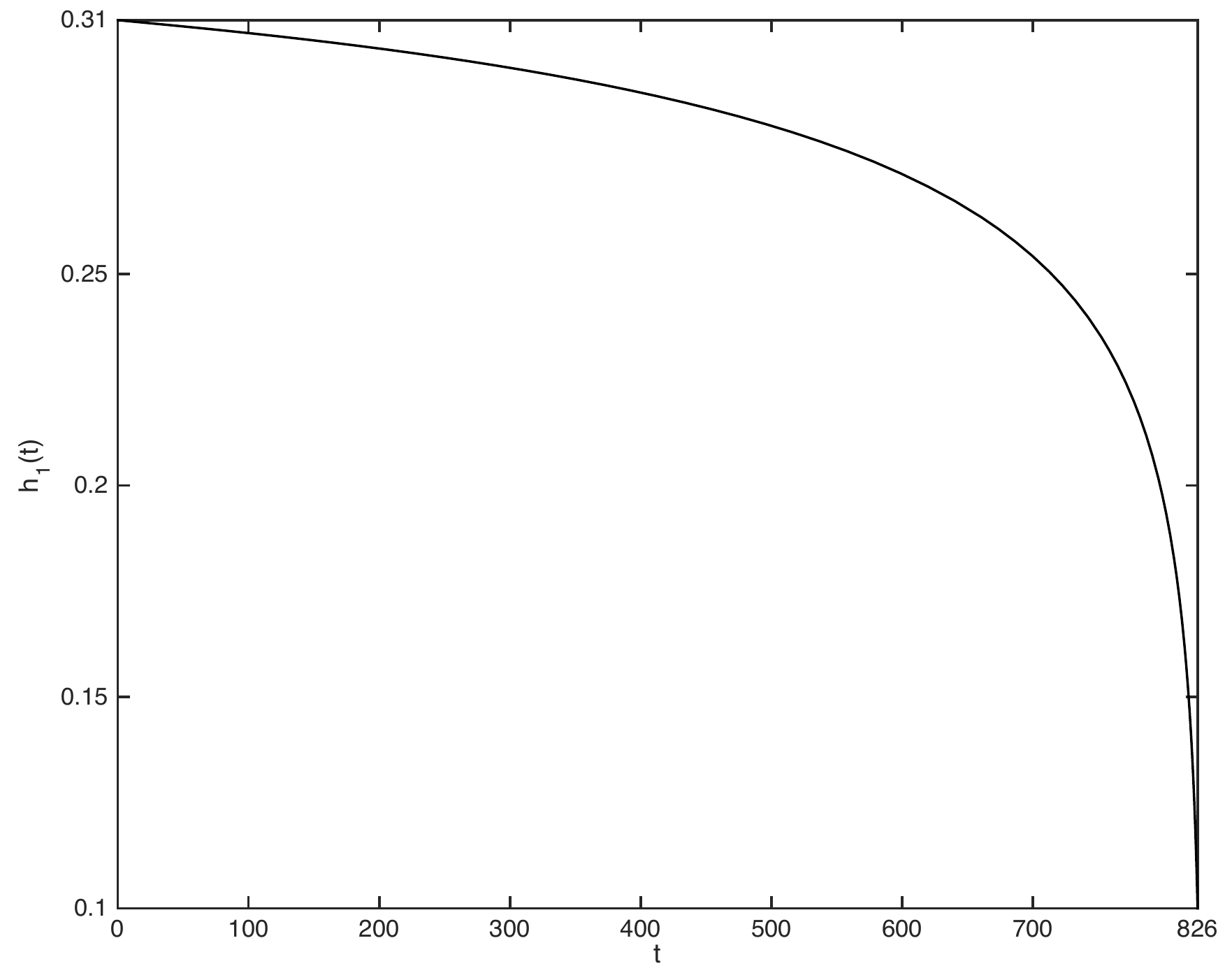}
\caption{The graph of $h_1(t)$ for $\e=0.07$ in the case of systems \eqref{eq:h_i-CahnHilliard} (left) and \eqref{eq:h2layers} (right) with $\tau=50$.}
\label{fig:h1}
\end{figure}

We see that the greater $\tau$ is, the slower the movement of the layers is. 
In particular, in Figure \ref{fig:h1} we see that the behavior of $h_1$ is the same, but the time taken for $h_1$ to reach the position $0.1$ is greater in the case 
of system \eqref{eq:h2layers} with $\tau=50$.

Now, we consider an example with $6$ layers.
For $\e=0.008$, in Tables \ref{table:6layers-tau0} and \ref{table:6layers-tau125} we numerically compute the difference $h_i(t)-h_i(0)$ for $i=1,\dots,6$ 
in the case $\tau=0$ and $\tau=\e^{-1}=125$, respectively,
and, in particular, we see that in the case $\tau=125$ the layers move slower than in the case without inertial terms.

\begin{table}[h!]
\begin{center}
\begin{tabular}{|c|c|c|c|c|}
\hline $s_i(t)$ &   $t=10^2$  &   $t=10^4$ &  $t=10^5$ & $t=1.55*10^5$  \\ 
\hline $s_1(t)$ & $2.99*10^{-7}$ & $3.00*10^{-5}$ & $3.13*10^{-4}$ & $4.96*10^{-4}$  \\
\hline $s_2(t)$ & $2.13*10^{-6}$ & $2.19*10^{-4}$ & $3.27*10^{-3}$ & $1.36*10^{-2}$ \\
\hline $s_3(t)$ & $1.54*10^{-6}$ & $1.60*10^{-4}$ & $2.64*10^{-3}$ & $1.25*10^{-2}$ \\ 
\hline $s_4(t)$ & $-2.03*10^{-6}$ & $-2.09*10^{-4}$ & $-3.09*10^{-3}$ & $-1.26*10^{-2}$ \\
\hline $s_5(t)$ & $-1.79*10^{-6}$ & $-1.85*10^{-4}$ & $-2.82*10^{-3}$ & $-1.21*10^{-2}$ \\
\hline $s_6(t)$ & $-4.76*10^{-8}$ & $-4.75*10^{-6}$ & $-4.62*10^{-5}$ & $-6.99*10^{-5}$ \\
\hline
\end{tabular}
\caption{The numerical computation of $s_i(t)=h_i(t)-h_i(0)$ in the case of system \eqref{eq:h_i-CahnHilliard} for $\e=0.008$. 
The initial positions of the layers are $0.18, 0.32, 0.45, 0.57, 0.71, 0.86$.}
\label{table:6layers-tau0}
\end{center}
\end{table}

\begin{table}[h!]
\vskip0.2cm
\begin{center}
\begin{tabular}{|c|c|c|c|c|}
\hline $s_i(t)$ &   $t=10^2$  &   $t=10^4$ &  $t=10^5$ & $t=1.55*10^5$  \\ 
\hline $s_1(t)$ & $2.99*10^{-7}$ & $3.00*10^{-5}$ & $3.13*10^{-4}$ & $4.94*10^{-4}$  \\
\hline $s_2(t)$ & $2.13*10^{-6}$ & $2.19*10^{-4}$ & $3.26*10^{-3}$ & $1.27*10^{-2}$ \\
\hline $s_3(t)$ & $1.54*10^{-6}$ & $1.60*10^{-4}$ & $2.63*10^{-3}$ & $1.17*10^{-2}$ \\ 
\hline $s_4(t)$ & $-2.03*10^{-6}$ & $-2.09*10^{-4}$ & $-3.08*10^{-3}$ & $-1.18*10^{-2}$ \\
\hline $s_5(t)$ & $-1.79*10^{-6}$ & $-1.84*10^{-4}$ & $-2.81*10^{-3}$ & $-1.13*10^{-2}$ \\
\hline $s_6(t)$ & $-4.76*10^{-8}$ & $-4.75*10^{-6}$ & $-4.62*10^{-5}$ & $-7.09*10^{-5}$ \\
\hline
\end{tabular}
\caption{The numerical computation of $s_i(t)=h_i(t)-h_i(0)$ in the case of system \eqref{eq:hN+1layers}.   
The values of the parameters are $\e=0.008$ and $\tau=125$; the initial positions of the layers are $0.18, 0.32, 0.45, 0.57, 0.71, 0.86$.}
\label{table:6layers-tau125}
\end{center}
\end{table}

In the previous computations we choose the same initial velocities for \eqref{eq:h_i-CahnHilliard}-\eqref{eq:hN+1layers} and 
the only difference is that in the case of system \eqref{eq:hN+1layers} the layers move slower than \eqref{eq:h_i-CahnHilliard}.
On the other hand, choosing different initial velocities, according to \eqref{eq:initial-h}, we can observe different dynamics. 
For instance, in the case of system \eqref{eq:hN+1layers} the points can change direction:
in Table \ref{table:change} we consider the same $\e$, $\tau$ and initial positions of the Table \ref{table:6layers-tau125}, but with opposite initial velocities,
namely, we choose $\ut_1=-\mathcal L(\ut_0)$ in \eqref{eq:initial-h}.
We see that the points change direction and after that we have the same behavior of Table \ref{table:6layers-tau125}.

\begin{table}[h!]
\vskip0.2cm
\begin{center}
\begin{tabular}{|c|c|c|c|c|c|}
\hline $s_i(t)$ &   $t=10^2$  & $t=2*10^2$ &  $t=10^4$ &  $t=10^5$ & $t=1.55*10^5$  \\ 
\hline $s_1(t)$ & $-0.11*10^{-6}$ & $1.4*10^{-9}$ & $2.93*10^{-5}$ & $3.12*10^{-4}$ & $4.94*10^{-4}$  \\
\hline $s_2(t)$ & $-0.80*10^{-6}$ & $9.7*10^{-9} $ & $2.14*10^{-4}$ & $3.25*10^{-3}$ & $1.22*10^{-2}$ \\
\hline $s_3(t)$ & $-0.58*10^{-6}$ & $7*10^{-9}$ & $1.56*10^{-4}$ & $2.62*10^{-3}$ & $1.12*10^{-2}$ \\ 
\hline $s_4(t)$ & $0.76*10^{-6}$ & $-9.3*10^{-9}$ & $-2.04*10^{-4}$ & $-3.07*10^{-3}$ & $-1.14*10^{-2}$ \\
\hline $s_5(t)$ & $0.67*10^{-6}$ & $-8.2*10^{-9}$ & $-1.80*10^{-4}$ & $-2.80*10^{-3}$ & $-1.09*10^{-2}$ \\
\hline $s_6(t)$ & $0.02*10^{-6}$ & $-2.3*10^{-10}$ & $-4.63*10^{-6}$ & $-4.61*10^{-5}$ & $-7.06*10^{-5}$ \\
\hline
\end{tabular}
\caption{In this table we consider the same initial positions and the same values of $\e$ and $\tau$ of Table \ref{table:6layers-tau125},  
but initial velocities with opposite sign respect to Table \ref{table:6layers-tau125}.}
\label{table:change}
\end{center}
\end{table}

We conclude this paper by comparing the solutions to systems \eqref{eq:hN+1layers} and \eqref{eq:h_i-CahnHilliard} as $\tau\to0^+$.
Let us rewrite system \eqref{eq:hN+1layers-compact} in the form
\begin{equation}\label{eq:system-h-h'}
	\begin{cases}
		\bm h'=\bm\eta,\\
		\tau\bm\eta'=\bm{\mathcal P}(\bm h)-\bm\eta-\tau\bm{\mathcal Q}(\bm h,\bm\eta),
	\end{cases}
\end{equation}
and system \eqref{eq:h_i-CahnHilliard} in the form
\begin{equation}\label{eq:h-CahnHilliard}
	\begin{cases}
		\bm h'=\bm\eta,\\
		\bm\eta=\bm{\mathcal P}(\bm h),
	\end{cases}
\end{equation}
where $\bm{\mathcal P}$ and $\bm{\mathcal Q}$ are defined in \eqref{eq:P(h)} and \eqref{eq:Q(h)}.
Notice that the functions $\bm{\mathcal P}$ and $\bm{\mathcal Q}$ are not well defined when $h_j=h_{j+1}$ for some $j$, 
but here we are interested in studying the system \eqref{eq:system-h-h'} when $l_j(t)>\delta$ for any $t\in[0,T]$ and any $j$ for some positive $\delta$ and $T$,
because system \eqref{eq:system-h-h'} describes the movement of the transition points when they are well separated for the hyperbolic Cahn--Hilliard equation \eqref{eq:hyp-CH}.
Therefore, in the following we consider system \eqref{eq:system-h-h'} for $t\in[0,T]$ where $T$ is such that $l_j(t)>\delta>0$ for any $t\in[0,T]$ and any $j\in\{1,N+2\}$.
%where $T_\e$ is the maximal time that the solution of \eqref{eq:hyp-CH} stays in the slow channel, see Theorem \ref{thm:main}.
Denote by $(\bm h,\bm\eta)$ the solutions to \eqref{eq:system-h-h'}  and $(\bm h_c,\bm\eta_c)$ the solutions of \eqref{eq:h-CahnHilliard},
and set
\begin{equation*}
	\mathcal E_\tau(t):=|\bm h(t)-\bm h_c(t)|+\tau|\bm\eta(t)-\bm\eta_c(t)|. 
\end{equation*}
A general theorem of Tihonov on singular perturbations can be applied to systems \eqref{eq:system-h-h'}-\eqref{eq:h-CahnHilliard} to prove that
if $(\bm h,\bm\eta)$ is a bounded solution of \eqref{eq:system-h-h'} for $t\in[0,T]$ and $\mathcal E_\tau(0)\rightarrow0$ as $\tau\rightarrow0$, 
then $\bm h\to\bm h_c$ uniformly in $[0,T]$ and $\bm\eta\to\bm\eta_c$ uniformly in $[t_1,T]$ for any $t_1>0$ as $\tau\to0^+$.

\begin{prop}\label{prop:tau0}
Fix $\e,\rho$ satisfying \eqref{eq:triangle} with $\e_0$ sufficiently small.
Let $(\bm h,\bm \eta)$ be a solution of \eqref{eq:system-h-h'} and $(\bm h_c,\bm\eta_c)$ a solution of \eqref{eq:h-CahnHilliard},
with $ \bm h(t),\bm h_c(t)\in\Omega_\rho$ for any $t\in[0,T]$.
Then, there exists $C>0$ (independent of $\tau$) such that 
\begin{equation}\label{E(t)<}
	\mathcal E_\tau(t)\leq C(\mathcal E_\tau(0)+\tau), \qquad \quad \mbox{ for } t\in[0,T].
\end{equation}
Moreover,
\begin{align}
	\int_0^T|\bm\eta(t)-\bm\eta_c(t)|dt &\leq C(\mathcal E_\tau(0)+\tau), \label{eta-L1}\\
	|\bm\eta(t)-\bm\eta_c(t)| &\leq C(\mathcal E_\tau(0)+\tau), \qquad \quad \mbox{ for } t\in[t_1,T],\label{eta-inf}
\end{align}
for all $t_1\in(0,T)$. 

In particular, from \eqref{E(t)<}, \eqref{eta-L1} and  \eqref{eta-inf}, it follows that, if $\mathcal E_\tau(0)\rightarrow0$ as $\tau\rightarrow0$, then
\begin{equation*}
	\lim_{\tau\rightarrow0}\sup_{t\in[0,T]}|\bm h(t)-\bm h_c(t)|
	=\lim_{\tau\rightarrow0} \int_0^T|\bm\eta(t)-\bm\eta_c(t)|dt
	=\lim_{\tau\rightarrow0}\sup_{t\in[t_1,T]}|\bm\eta(t)-\bm\eta_c(t)|=0,
 \end{equation*}
for any $t_1\in(0,T)$.
\end{prop}

\begin{proof}
For $t\in[0,T]$, define
\begin{equation*}
	\bm\delta_{\bm h}(t):=\bm h(t)-\bm h_c(t), \qquad \quad
	\bm\delta_{\bm\eta}(t):=\bm\eta(t)-\bm\eta_c(t).
\end{equation*}
Since $\bm h(t),\bm h_c(t)\in\Omega_\rho$ for $t\in[0,T]$, by using Proposition \ref{prop:alfa,beta} and using that $l_j>\delta>0$, we get
\begin{equation}\label{eq:P-Q-stime}
	\begin{aligned}
		|\mathcal{P}(\bm h_c)|&\leq \frac{C}{\delta}\exp(-A\delta/\e), \qquad \qquad\quad & |J\mathcal{P}(\bm h_c)|&\leq \frac{C}{\e^2\delta^2}\exp(-A\delta/\e), \\
		|\mathcal{P}(\bm h_c+\bm\delta_{\bm h})-\mathcal{P}(\bm h_c)|&\leq \frac{C}{\e^2\delta^2}\exp(-A\delta/\e)|\bm\delta_{\bm h}|, & |\mathcal Q(\bm h,\bm\eta)|&\leq \frac{C}{\delta}|\bm\eta|^2.
	\end{aligned}
\end{equation}
for all $t\in[0,T]$. 
Here and in what follows, $C$ is a positive constant independent of $\tau$ whose value may change from line to line. 
We have
\begin{equation*}
	\bm\delta_{\bm h}'=\bm\eta-\bm\eta_c, \qquad \quad 
	\tau\bm\delta_{\bm\eta}'=\mathcal{P}(\bm h_c+\bm\delta_c)-\mathcal{P}(\bm h_c)-\bm\delta_{\bm\eta}-\tau\mathcal Q(\bm h,\bm\eta)-\tau J\mathcal{P}(\bm h_c)\mathcal{P}(\bm h_c).
\end{equation*}
Since $\displaystyle\frac d{dt}|\bm\delta|=\frac{\bm\delta'\cdot\bm\delta}{|\bm\delta|}$ for any $\bm\delta(t)\in\mathbb{R}^{N+1}$, 
using estimates \eqref{eq:P-Q-stime} and Cauchy--Schwarz inequality, we obtain
\begin{equation*}
	\frac d{dt}|\bm\delta_{\bm h}|\leq|\bm\delta_{\bm\eta}|, \qquad 
	\tau\frac d{dt}|\bm\delta_{\bm\eta}|\leq C|\bm\delta_{\bm h}|-|\bm\delta_{\bm\eta}|+C\tau.
\end{equation*}
Summing, one has
\begin{equation*}
	\frac d{dt}\left(|\bm\delta_{\bm h}|+\tau|\bm\delta_{\bm\eta}|\right)\leq C|\bm\delta_{\bm h}|+C\tau,
\end{equation*}
and so, 
\begin{equation}\label{d/dt E<}
	\frac d{dt}\mathcal E_\tau(t)\leq C\left(\mathcal{E}_\tau(t)+\tau\right),  \qquad \quad \mbox{ for } t\in[0,T].
\end{equation}
Integrating \eqref{d/dt E<} and applying Gr\"onwall's Lemma, we obtain \eqref{E(t)<}. 
In particular, from \eqref{E(t)<}, it follows that 
\begin{equation}\label{delta_h}
	|\bm\delta_{\bm h}(t)|\leq C(\mathcal E_\tau(0)+\tau),  \qquad \quad\qquad \mbox{ for } t\in[0,T].
\end{equation}
Substituting \eqref{delta_h} into the equation for $\bm\delta_{\bm\eta}$, we obtain
\begin{equation*}
	\tau\frac{d}{dt}|\bm\delta_{\bm\eta}|\leq-|\bm\delta_{\bm\eta}|+C(\mathcal E_\tau(0)+\tau),
\end{equation*}
and integrating the latter estimate we infer \eqref{eta-L1}; moreover, we have
\begin{equation*}
	\frac{d}{dt}\left(\tau e^{t/\tau}|\bm\delta_{\bm\eta}(t)|\right)\leq C(\mathcal E_\tau(0)+\tau)e^{t/\tau},
\end{equation*}
and so
\begin{equation*}
	|\bm\delta_{\bm\eta}(t)| \leq C(\mathcal E_\tau(0)+\tau) +\mathcal E_\tau(0)\frac{e^{-t/\tau}}{\tau},
\end{equation*}
for $t\in[0,T]$.
Therefore, for any fixed $t_1\in(0,T)$, we obtain \eqref{eta-inf}.
\end{proof}

\section*{Acknowledgments} 
This is a pre-print of an article published in Journal of Dynamics and Differential Equations. 
The final authenticated version is available online at: https://doi.org/10.1007/s10884-019-09806-6.
We thank the anonymous referee for the careful review and for the comments which helped us to improve the paper.

\end{document}